\documentclass{amsart}
\usepackage{graphicx}
\usepackage{amsmath}
\usepackage{amssymb}
\usepackage{color}
\usepackage{amsfonts}
\usepackage{mathrsfs}
\usepackage{setspace} 
\newtheorem{theorem}{Theorem}[section]
\newtheorem{corollary}[theorem]{Corollary}

\newtheorem{lemma}[theorem]{Lemma}
\newtheorem{proposition}[theorem]{Proposition}
\theoremstyle{definition}
\newtheorem{e-definition}[theorem]{Definition}
\theoremstyle{remark}
\newtheorem{remark}[theorem]{Remark}

\numberwithin{equation}{section}
\DeclareMathOperator{\Ker}{\mathrm{Ker}}
\DeclareMathOperator{\Ima}{\mathrm{Im}}
\DeclareMathOperator{\ord}{\mathrm{ord}}
\DeclareMathOperator{\Rea}{\mathrm{Re}}
\DeclareMathOperator{\diag}{\mathrm{diag}}
\DeclareMathOperator{\adj}{\mathrm{adj}}

\begin{document}


\keywords{Quaternions, Wiener algebras, Wiener-Hopf factorization, Riemann-Hilbert problem, difference equations, convolution equations, rational matrix functions}
\subjclass{13J05,15A09,30E25,47B35,\\47B39,47S10}


\title{Quaternionic Wiener algebras, factorization and applications}
\date{2016-12-14}

\author[Y. Shelah]{Yonatan Shelah}
\address{Department of Mathematics, University of Michigan and School of Mathematics, Tel Aviv University} 
\email{yonshel@umich.edu}
\date{December 14, 2016}%

\thanks{This paper is based on the author's MSc thesis, which was written when he was a student at Tel Aviv University. The author would like to thank his advisor, Prof. Daniel Alpay, for introducing him to quaternionic analysis and proposing which topics to tackle. His input and encouragement were most valuable.}

\begin{abstract}
We define an almost periodic extension of the Wiener algebras in the quaternionic setting and prove a Wiener-L\'evy type theorem for it, as well as extending the theorem to the matrix-valued case. We prove a Wiener-Hopf factorization theorem for the quaternionic matrix-valued Wiener algebras (discrete and continuous) and explore the connection to the Riemann-Hilbert problem in that setting. As applications, we characterize solvability of two classes of quaternionic functional equations and give an explicit formula for the canonical factorization of quaternionic rational matrix functions via realization.
\end{abstract}

\maketitle
\pagenumbering{gobble}
\pagenumbering{arabic}

\tableofcontents

\section{Introduction and preliminaries}
Our aim is an extension of certain theorems from complex analysis and operator theory to the realm of quaternions. The Wiener algebra has two well-known variants: the discrete one (defined on the unit sphere) and the continuous one (defined on the real line). Both have been extended to the quaternionic (scalar-valued) setting by D. Alpay et al. (see \cite{acks}), successfully carrying over the Wiener-L\'evy theorem. We seek to gain more richness about these quaternionic algebras, namely Wiener-Hopf factorization and the Riemann-Hilbert problem in the matrix-valued case. 

Let us recall that the (complex) discrete Wiener algebra $\mathcal W^{n\times n}$ consists of functions of the form 
$$
F(t)=\sum_{u\in\mathbb Z}F_ue^{iut},
$$ 
where $F_u\in\mathbb C^{n\times n}$ and $\sum_{u\in\mathbb Z}\|F_u\|<\infty$, $\|\cdot\|$ denoting the operator norm. The space $\mathcal W^{n\times n}$ with pointwise multiplication and the norm defined above is a Banach algebra. When $n=1$, we denote the Wiener algebra by $\mathcal W$.\\
The Wiener-L\'evy theorem characterizes the invertible elements of $\mathcal W^{n\times n}$: an element is invertible in $\mathcal W^{n\times n}$ if and only if it is pointwise invertible (in $\mathbb C^{n\times n}$).

A quaternionic valued function $f(p)$ belongs to the (discrete) quaternionic Wiener algebra $\mathcal W_{\mathbb H}$ if it is of the form $\sum_{u\in\mathbb Z}p^u f_u$, where $\sum_{u\in\mathbb Z} |f_u|<\infty$. The sum of two elements in $\mathcal W_{\mathbb H}$ is defined in the natural way, while their product, denoted by $\star$, is obtained by taking the convolution of the coefficients (much like the $\star$ product for polynomials and more generally slice hyperholomorphic functions).   In \cite{acks}, a counterpart of the Wiener-L\'evy theorem was successfully obtained in the aforementioned scalar-valued case. Moreover, the theorem characterized invertibility via a slice of the function with an arbitrary complex plane.

The continuous Wiener algebra $\mathcal W(\mathbb R,\mathbb C^{n\times n})$ consists of functions of the form 
$$
G(t)=C+\int_{\mathbb R}K(u)e^{itu}\, du,
$$
where $C\in\mathbb C^{n\times n}$, $K\in L_1(\mathbb R,\mathbb C^{n\times n})$. There is a Wiener-L\'evy theorem in this case, too, with invertibility being equivalent to the condition 
$$
\inf_{t\in\mathbb R}|\det(f(t))|>0.
$$ 
The scalar-valued case was also extended in \cite{acks} to the quaternionic setting. 

The initial motivation behind this paper was the desire to extend the connection between the Wiener algebras and certain functional equations (of three types: difference, convolution and more generally integro-difference) from the complex setting to the quaternionic one. Consider (over $\mathbb C$) the integro-difference equation
\begin{equation}
\sum_{r\in\mathbb R}\phi(t-r)a_r+\int_{0}^{\infty}\phi(s)k(t-s)ds=f(t),
\end{equation}
where $f\in L_p(0,\infty)$ is given, $\phi\in L_p(0,\infty)$ is sought (we define $\phi(t)=0$ for $t\leq 0$), $k\in L_1(\mathbb R)$, and the $a_r\in\mathbb C$ vanish for $r$ outside of a countable set and satisfy $\sum_{r\in\mathbb R}|a_r|<\infty$. \cite{gf} includes a study of the solvability of (1.1) via the factorization of the symbol function 
$$
\mathcal A(t)=\sum_{r\in\mathbb R}a_re^{irt}+\int_{\mathbb R}k(u)e^{iut}du.
$$
The collection of such functions is denoted by $\mathcal B(\mathbb R)$. Note that functions of the form $\sum_{r\in\mathbb R}a_re^{irt}$ are almost periodic, and their collection is called the almost periodic Wiener algebra, being denoted by $\mathcal APW$. 

The first section of the paper extends the above definitions to quaternions, yet again carrying over the Wiener-L\'evy theorem. In this general framework of Wiener algebras, we also characterize invertibility in the matrix-valued case.

The second section deals with factorization in the quaternionic case, but we limit the discussion to the (matrix-valued) discrete and continuous Wiener algebras. The reason for this limitation is that, in fact, it is known that even in the complex case both $\mathcal B(\mathbb R)^{n\times n}$ and $\mathcal APW^{n\times n}$ fail to ensure factorization in general (that is, for any invertible function) for $n>1$. We should note that the case of $n=1$ is interesting since factorization does hold for it over $\mathbb C$, but our tools are not enough to determine if this remains the case over $\mathbb H$. Regardless, what is notable about the discrete and continuous algebras, is the connection between factorization and the Riemann-Hilbert problem.  We will show that the theorems from the complex setting successfully carry over.

The final two sections deal with applications of factorization: To difference and convolution equations (but not integro-difference ones due to the lack of a factorization theorem) and to canonical factorization of quaternionic rational matrix functions. Again, the results are direct counterparts of the known theorems in the complex cases.

We follow \cite{css} for the following definitions. We denote by $\mathbb S$ the sphere of unitary purely imaginary quaternions. Any element $i\in\mathbb S$ satisfies $i^2=-1$, and using $e_1,e_2,e_3$ as a standard basis, we get
$$
\mathbb S=\{xe_1+ye_2+ze_3:x,y,z\in\mathbb R,x^2+y^2+z^2=1\}.
$$ 
Any two orthogonal elements $i,j\in\mathbb S$ (in the sense $ij=-ji$) form a new basis $i,j, ij$ of $\mathbb H$. Given a quaternion $p_0$, it determines a sphere $[p_0]$ consisting of all the points of the form $q^{-1}p_0q$ for $q\not=0$.\\ 
In some cases, it is useful to use a map which transforms a quaternion into a $2\times2$ complex matrix or, more generally, a quaternionic matrix into a complex matrix of double size (for example, see \cite{z}). This will appear throughout the work (albeit often in the context of a different map to be defined later).

\begin{e-definition}
Let $i\perp j\in\mathbb S$. Any $p\in\mathbb H$ has a unique representation \newline $p=a+bj$, where $a,b\in\mathbb C_i:=\mathbb R+i\mathbb R$. More generally, any $P\in\mathbb H^{n\times n}$ has a unique representation $P=A+Bj$, where $A,B\in\mathbb C_i^{n\times n}$. We define a map $\chi_{i,j}$, which maps $\mathbb H^{n\times n}$ into $\mathbb C_i^{2n\times 2n}$ for any $n\in\mathbb N$, via
$$
\chi_{i,j}(P)=\left[\begin{matrix} 
A &B \\
-\overline{B} &\overline{A}
\end{matrix}\right].
$$
\end{e-definition}

We will write $\chi$ instead of $\chi_{i,j}$ whenever $i,j$ are clear from the context. Note that $\chi$ is injective, additive and multiplicative, hence its utility. The image $\chi(\mathbb H^{n\times n})$ is the set of all matrices $Q\in\mathbb C_i^{2n\times 2n}$ that satisfy the relation $J_n\overline QJ_n^T=Q$, where
$$
J_n=\left[\begin{matrix}
0 &I_n \\
-I_n &0
\end{matrix}\right].
$$

\section{An almost periodic extension of the quaternionic Wiener algebra}

We begin with the scalar-valued case. The following definition is analogous to the one for the discrete Wiener algebra (see \cite{acks}).

\begin{e-definition}
We denote by $\mathcal APW(\mathbb S\mathbb R,\mathbb H)$ the set of functions of the form
\begin{equation}\label{effe}
F(it)=\sum_{u\in\mathbb R}e^{itu}f_u
\end{equation}
where the $f_u$ are quaternions vanishing for all but a countable subset of $u \in \mathbb{R}$, and
\[
\sum_{u\in\mathbb R}|f_u|<\infty.
\]
This set can be endowed with the multiplication
$$
(F\star G)(it)=\sum_{u\in\mathbb R}e^{itu}\sum_{v\in\mathbb R}f_vg_{u-v}.
$$
We also define
$$
\| F\|= \sum_{u\in\mathbb R}|f_u|.
$$
\end{e-definition}

\begin{e-definition}
We denote by $\mathcal B(\mathbb S\mathbb R,\mathbb H)$ the sum of $\mathcal APW(\mathbb S\mathbb R,\mathbb H)$ and $\mathcal W(\mathbb S\mathbb R,\mathbb H)$, which is to say the set of all functions of the form
\begin{equation}\label{effe}
F(it)=\sum_{u\in\mathbb R}e^{itu}f_u+\int_{\mathbb R}e^{itu}\phi_f(u)\, du
\end{equation}
where $\phi_f\in L_1(\mathbb R,\mathbb H)$ and (again) the $f_u$ are quaternions satisfying
\[
\sum_{u\in\mathbb R}|f_u|<\infty.
\]
This set can be endowed with the multiplication
\begin{equation}
\begin{split}
(F\star G)(it)=\sum_{u\in\mathbb R}e^{itu}\sum_{v\in\mathbb R}f_vg_{u-v}+\int_{\mathbb R}e^{itu}du\sum_{v\in\mathbb R}f_v\phi_g(u-v) \\
+\int_{\mathbb R}e^{itu}du\sum_{v\in\mathbb R}\phi_f(u-v)g_v+\int_{\mathbb R}e^{itu} (\phi_f\circ \phi_g)(u)\, du,
\end{split}
\end{equation}
where $\circ$ denotes convolution. We also define
$$
\| F\|= \sum_{u\in\mathbb R}|f_u|+\int_{\mathbb R}|\phi_f(u)| du.
$$
\end{e-definition}

In \cite{acks} it was shown that the discrete and continuous Wiener algebras are real Banach algebras. Along the same lines, we have:

\begin{proposition}
$\mathcal B(\mathbb S\mathbb R,\mathbb H)$ endowed with the $\star$-multiplication is a real Banach algebra, which contains $\mathcal APW(\mathbb S\mathbb R,\mathbb H)$ and $\mathcal W(\mathbb S\mathbb R,\mathbb H)$ as closed subalgebras.
\end{proposition}

\begin{proof}
It is clear that $\mathcal B(\mathbb S\mathbb R,\mathbb H)$ is a real algebra, and that $\mathcal APW(\mathbb S\mathbb R,\mathbb H)$ and $\mathcal W(\mathbb S\mathbb R,\mathbb H)$ are subalgebras. Now let $(F_n)_{n\in\mathbb N}$ be a Cauchy sequence in $\mathcal B(\mathbb S\mathbb R,\mathbb H)$. Writing 
$$
F_n(it)=\sum_{u\in\mathbb R}e^{itu}f_{n,u}+\int_{\mathbb R}e^{itu}\phi_{f,n}(u)\, du,
$$ 
we see that 
$$
(\sum_{u\in\mathbb R}e^{itu}f_{n,u})_{n\in\mathbb N} \mbox{ and } (\int_{\mathbb R}e^{itu}\phi_{f,n}(u)\, du)_{n\in\mathbb N}
$$ 
are Cauchy sequences in $\mathcal APW(\mathbb S\mathbb R,\mathbb H)$ and $\mathcal W(\mathbb S\mathbb R,\mathbb H)$, respectively. Since it is known that $\mathcal W(\mathbb S\mathbb R,\mathbb H)$ is a Banach algebra (hence a closed subalgebra), we only need to show that the same holds for $\mathcal APW(\mathbb S\mathbb R,\mathbb H)$. The set $$\{u\in\mathbb R:\exists n\in\mathbb N . f_{n,u}\neq 0\}$$ is countable, so we may enumerate it as $(w_m)_{m\in\mathbb N}$. For every $m\in\mathbb N$ we have that $(f_{n,w_m})_{n\in\mathbb N}$ is a Cauchy sequence of quaternions. Thus $\lim_{n\to\infty}f_{n,w_m}$ exists and we denote it by $f_{w_m}$. It is not hard to see that $\sum_{m\in\mathbb N}e^{itw_m}f_{w_m}$ is the norm limit of $\sum_{u\in\mathbb R}e^{itu}f_{n,u}$ as $n\to\infty$ (much like the proof that $l_1$ is a Banach space). Finally, $\mathcal B(\mathbb S\mathbb R,\mathbb H)$ is a Banach algebra:
\begin{equation*}
\begin{split}
\|F\star G\|&=\sum_{u\in\mathbb R}|\sum_{v\in\mathbb R}f_vg_{u-v}|+\int_{\mathbb R}|\sum_{v\in\mathbb R}f_v\phi_g(u-v)\\
&+\sum_{v\in\mathbb R}\phi_f(u-v)g_v+(\phi_f\circ \phi_g)(u)|du+\leq\sum_{u\in\mathbb R}\sum_{v\in\mathbb R}|f_v||g_{u-v}|\\
&+\int_{\mathbb R}\sum_{v\in\mathbb R}|f_v||\phi_g(u-v)|+\sum_{v\in\mathbb R}|\phi_f(u-v)||g_v|+(|\phi_f|\circ |\phi_g|)(u)du\\
&=\|F\|\|G\|.
\end{split}
\end{equation*}
\end{proof}

In \cite{acks}, a map $\omega$ was introduced as a way of viewing the quaternionic Wiener algebras as complex matrix-valued Wiener algebras. Using the same symbol, we analogously define a map $\omega=\omega_{i,j}$ depending on the choice of an imaginary unit $i\in\mathbb S$ and of a $j\in\mathbb S$
orthogonal to $i$, that is, satisfying $ij=-ji$.

\begin{e-definition}
Let $F(it)=\sum_{u\in\mathbb R}e^{itu}f_u+\int_{\mathbb R}e^{itu}\phi_f(u)\, du$. Then
\[
\omega_{i,j}(F)(t)=\sum_{u\in\mathbb R}e^{itu}\chi(f_u)+\int_{\mathbb R}e^{itu}\chi(\phi_f(u)) du,
\]
where $\chi=\chi_{i,j}$ is defined as in the introduction.
\end{e-definition}
We will write $\omega=\omega_{i,j}$ if the context is clear. It is immediate that the entries of $\omega(F)(t)$ belong to $\mathcal B(\mathbb R,\mathbb C_i)$, and moreover, $\omega$ maps $\mathcal B(\mathbb S\mathbb R,\mathbb H)$ injectively into $\mathcal B^{2\times 2}(\mathbb R,\mathbb C_i)$ with values in $\mathbb C_i^{2\times 2}$. Likewise, it maps $\mathcal APW(\mathbb S\mathbb R,\mathbb H)$ injectively into $\mathcal APW^{2\times 2}(\mathbb R,\mathbb C_i)$. It can also be verified (using the properties of $\chi$) that a function $G\in\mathcal B^{2\times 2}(\mathbb R,\mathbb C_i)$ (or a subalgebra) is in the image under $\omega$ (or of the respective subalgebra) if and only if $J_n\overline{G(-t)}J_n^T=G(t)$.

As explained in \cite{acks}, the reason for using $\omega$ is that $\chi$ generally does not respect $\star$-multiplication (only pointwise multiplication), while $\omega$ does. The proof barely changes under our extension, but we include it for the sake of completeness.

\begin{lemma}\label{lemmastar}
Let $F,G\in\mathcal B(\mathbb S\mathbb R, \mathbb H)$. Then
$$
\omega(F\star G)(t)=\omega(F)(t)\omega(G)(t), \quad t \in \mathbb{R}.
$$
\end{lemma}

\begin{proof}
Using the properties of $\chi$, we get
\begin{equation*}
\begin{split}
\omega(F\star G)(t)&=\sum_{u\in\mathbb R}e^{itu}\sum_{v\in\mathbb R}\chi(f_vg_{u-v})+\int_{\mathbb R}e^{itu}du\sum_{v\in\mathbb R}\chi(f_v\phi_g(u-v))\\
&+\int_{\mathbb R}e^{itu}du\sum_{v\in\mathbb R}\chi(\phi_f(u-v)g_v)+\int_{\mathbb R}e^{itu} \chi(\phi_f\circ \phi_g)(u)\, du\\
&=\sum_{u\in\mathbb R}e^{itu}\sum_{v\in\mathbb R}\chi(f_v)\chi(g_{u-v})+\int_{\mathbb R}e^{itu}du\sum_{v\in\mathbb R}\chi(f_v)\chi(\phi_g(u-v))\\
&+\int_{\mathbb R}e^{itu}du\sum_{v\in\mathbb R}\chi(\phi_f(u-v))\chi(g_v)+\int_{\mathbb R}e^{itu} (\chi(\phi_f)\circ \chi(\phi_g))(u)\, du\\
&=\omega(F)(t)\omega(G)(t).
\end{split}
\end{equation*}
\end{proof}

The following theorem is inspired by the theorems proved in \cite{acks} for the discrete and continuous cases. The main difference that needs to be accounted for is the fact that almost periodic functions cannot be continuously extended to $\infty$, so the domain being considered is not compact.

\begin{theorem}
Let $F\in\mathcal B(\mathbb S\mathbb R,\mathbb H)$. The following are equivalent{\rm :}
\begin{enumerate}
\item[(i)] The function $F$ is invertible in $\mathcal B(\mathbb S\mathbb R,\mathbb H)$.
\item[(ii)] There exist orthogonal $i,j\in\mathbb S$ such that $$\inf_{t\in\mathbb R}  |\det \omega_{i,j}(F)(t)|>0.$$
\item[(iii)] For any orthogonal $i,j\in\mathbb S,$ $$\inf_{t\in\mathbb R}  |\det \omega_{i,j}(F)(t)|>0.$$
\item[(iv)] $$\inf_{i\perp j\in\mathbb S, t\in\mathbb R}  |\det \omega_{i,j}(F)(t)|>0.$$
\item[(v)] $$\inf_{p\in\mathbb S\mathbb R}  |F(p)|>0.$$
\item[(vi)] For any $i\in\mathbb S,$ $$\inf_{t\in\mathbb R}  |F(it)|>0.$$
\end{enumerate}
\end{theorem}

For the proof, we need two lemmas, which are analogous to formulas proved in \cite{acks} (for the discrete and continuous cases).

\begin{lemma}\label{lemmastar}
Let $
F\in\mathcal B(\mathbb S\mathbb R, \mathbb H)$ 
be given by 
$$
F(it)=\sum_{u\in\mathbb R}e^{itu}f_u+\int_{\mathbb R}e^{itu}\phi_f(u)\, du.
$$ 
Then
$$
(F\star F^c)(it)=\det(\omega(F)(t)),
$$
where $F^c(it)=\sum_{u\in\mathbb R}e^{itu}\overline {f_u}+\int_{\mathbb R}e^{itu}\overline {\phi_f(u)}\, du$.
\end{lemma}

For the proof, we mimic the calculation in \cite{acks}, but it is longer due to there being two components.

\begin{proof}
Writing $f_u=a_u+b_uj$, $\phi_f(u)=k(u)+l(u)j$ where $a_u,b_u,k(u),l(u)$ are $\mathbb C_i$-valued, we get
\begin{equation*}
\begin{split}
(F\star F^c)(it)&=\sum_{u\in\mathbb R}e^{itu}\sum_{v\in\mathbb R}f_v\overline{f_{u-v}}+\int_{\mathbb R}e^{itu}du\sum_{v\in\mathbb R}f_v\overline{\phi_f(u-v)} \\
&+\int_{\mathbb R}e^{itu}du\sum_{v\in\mathbb R}\phi_f(u-v)\overline{f_v}+\int_{\mathbb R}e^{itu} (\phi_f\circ \overline{\phi_f})(u)\, du \\
&=\sum_{u\in\mathbb R}e^{itu}[\sum_{v\in\mathbb R}(a_v\overline{a_{u-v}}+b_v\overline{b_{u-v}})+\sum_{v\in\mathbb R}(-a_vb_{u-v}+b_va_{u-v})j]\\
&+2\int_{\mathbb R}e^{itu}du[\sum_{v\in\mathbb R}\Rea(a_v\overline{k(u-v)}+b_v\overline{l(u-v)})\\
&+\sum_{v\in\mathbb R}\Rea(-a_v\phi_f(u-v)+\phi_f(v)a_{u-v})]+\int_{\mathbb R}e^{itu} [(k\circ \overline{k}+l\circ \overline{l})(u)\\
&+(-k\circ l+l\circ k)(u)j]\, du\\
&=\sum_{u\in\mathbb R}e^{itu}\sum_{v\in\mathbb R}(a_v\overline{a_{u-v}}+b_v\overline{b_{u-v}})+2\int_{\mathbb R}e^{itu}du\sum_{v\in\mathbb R}\Rea(a_v\overline{k(u-v)}\\
&+b_v\overline{l(u-v)})+\int_{\mathbb R}e^{itu} (k\circ \overline{k}+l\circ \overline{l})(u)\, du.\\
\end{split}
\end{equation*}
Note that
$$
\omega(F)(t)=\left[\begin{matrix}
\sum_{u\in\mathbb R}e^{itu}a_u+\int_{\mathbb R}e^{itu}k(u)\, du &\sum_{u\in\mathbb R}e^{itu}b_u+\int_{\mathbb R}e^{itu}l(u)\, du\\
-\sum_{u\in\mathbb R}e^{itu}\overline{b_u}-\int_{\mathbb R}e^{itu}\overline{l(u)}\, du &\sum_{u\in\mathbb R}e^{itu}\overline{a_u}+\int_{\mathbb R}e^{itu}\overline{k(u)}\, du
\end{matrix}\right],
$$
and by carrying out the multiplication we see that $(F\star F^c)(it)=\det(\omega(F)(t))$.
\end{proof}

\begin{lemma}\label{lemmastar}
Let $F,G\in\mathcal B(\mathbb S\mathbb R, \mathbb H)$. Then
$$
\forall i \in \mathbb{S} \quad \forall t \in \mathbb{R} \quad \exists i_t\in \mathbb{S} \quad (F\star G)(it)=F(it)G(i_tt).
$$
\end{lemma}

\begin{proof}
If $F(it)\neq 0$, then let $i_t=F(it)^{-1}iF(it)$. Then
\begin{equation*}
\begin{split}
F(it)G(i_tt)&=F(it)(\sum_{u\in\mathbb R}e^{i_ttu}g_u+\int_{\mathbb R}e^{i_ttu}\phi_g(u)\, du)\\
&=F(it)(\sum_{u\in\mathbb R}(\cos(tu)+F(it)^{-1}iF(it)\sin(tu))g_u\\
&+\int_{\mathbb R}(\cos(tu)+F(it)^{-1}iF(it)\sin(tu))\phi_g(u)du)\\
&=\sum_{u\in\mathbb R}e^{itu}F(it)g_u+\int_{\mathbb R}e^{itu}F(it)\phi_g(u)\, du\\
&=(F\star G)(it).
\end{split}
\end{equation*}
If $F(it)=0$, let $\epsilon>0$ and define $F_{\epsilon}=F+\epsilon$. Then
$$
(F\star G)(it)=(F_{\epsilon}\star G)(it)-\epsilon G(it)=F_{\epsilon}(it)G(i_tt)-\epsilon G(it)=\epsilon(G(i_tt)-G(it)).
$$
Since $G$ is bounded and $\epsilon$ is arbitrary, we get $(F\star G)(it)=0$. In particular, $(F\star G)(it)=F(it)G(jt)$ for any $j\in \mathbb S$.
\end{proof}

Now let us prove Theorem 2.6. We follow the outline of the proof in \cite{acks} for the discrete case, modifying the arguments so as to deal with infimum values.

\begin{proof}
$(i) \implies (iv)$:
Denoting the inverse by $F^{-\star}$, we have by the multiplicativity of $\omega$ 
\begin{equation}
|\det \omega_{i,j}(F)(t)||\det \omega_{i,j}(F^{-\star})(t)|=1.
\end{equation}
Now, in general, for any $G\in\mathcal B(\mathbb S\mathbb R,\mathbb H)$ the matrix $\omega_{i,j}(G)(t)$ is of the form
\begin{equation}
\left[\begin{matrix} r(it) & s(it)\\
-\overline{s(-it)}& \overline{r(-it)}\end{matrix}\right]
\end{equation}
and thus
\begin{equation}
\begin{split}
\sup_{i,j,t} |\det \omega_{i,j}(G)(t)|&=\sup_{i,j,t} |\det\left[\begin{matrix} r(it) & s(it)\\
-\overline{s(-it)}& \overline{r(-it)}\end{matrix}\right]| \\
&\leq \sup_{i,j,t} |r(it)| \sup_{i,j,t} |r(-it)|+\sup_{i,j,t} |s(it)| \sup_{i,j,t} |s(-it)| \\
&\leq 2(\sup_{i,t}|G(it)|)^2\leq 2||G||^2,
\end{split}
\end{equation}
where we have used the fact that $|G(it)|^2=|r(it)|^2+|s(it)|^2$. In particular, $\sup_{i,j,t} |\det \omega_{i,j}(F^{-\star})(t)|\leq 2||F^{-\star}||^2$ and thus (2.4) yields
\begin{equation}
\inf_{i,j,t}|\det \omega_{i,j}(F)(t)|\geq\frac{1}{2||F^{-\star}||^2}.
\end{equation}
$(iv) \implies (iii) \implies (ii)$: Trivial. \\
$(ii) \implies (i)$: Let us assume that $$\inf_{t\in\mathbb R} |\det \omega_{i,j}(F)(t)|>0$$ for some orthogonal $i,j\in\mathbb S$. As a scalar complex function in the algebra $\mathcal B(\mathbb R,\mathbb C)$, $d(it):=\det \omega_{i,j}(F)(t)$ fulfills the condition for invertibility. Thus if we write
\begin{equation*}
\omega_{i,j}(F)(t)=\left[\begin{matrix} a(it) & b(it)\\
-\overline{b(-it)}& \overline{a(-it)}\end{matrix}\right], \\
\end{equation*}
then using the formula for inverting a $2 \times 2$ matrix we see that
\begin{equation*}
{\omega_{{i,j}}(F)}^{-1}(t)=\frac{1}{d(it)}\left[\begin{matrix} \overline{a(-it)} & -b(it)\\
\overline{b(-it)} & a(it)\end{matrix}\right] \\
\end{equation*}
is a matrix function in the algebra $\mathcal B(\mathbb R,\mathbb C^{2 \times 2})$. Let 
$$
H(it):=\frac{1}{d(it)}(\overline{a(-it)} -b(it) j).
$$ 
If $d(it)=\overline{d(-it)}$, then $\omega_{i,j}(H)(t)=\omega_{i,j}(F)^{-1}(t)$, which implies $H=F^{-\star}$. Indeed:
$$
d(it)=a(it)\overline{a(-it)}+b(it)\overline{b(-it)}=\overline{d(-it)}.
$$
$(v) \implies (iv)$: By Lemma 2.7, 
$$\det\omega_{i,j}(F)(t)=(F\star F^c)(it) \quad \forall i\perp j\in\mathbb S, t\in \mathbb R.$$ Using Lemma 2.8, we get
\begin{equation}
\begin{split}
\inf_{i,j,t}|\det\omega_{i,j}(F)(t)|&=\inf_{i,t}|(F\star F^c)(it)|=\inf_{i,t}|F(it)||F^c(i_tt)|\\
&\geq \inf_{p\in\mathbb S\mathbb R}|F(p)|\inf_{p\in\mathbb S\mathbb R}|F^c(p)|.
\end{split}
\end{equation}
Note that $F^c(\mathbb S\mathbb R)=\overline{F(\mathbb S\mathbb R)}$. To see this, write $F(it)=H(t)+iK(t)$, where
$$
H(t)=\sum_{u\in\mathbb R}\cos(tu)f_u+\int_{\mathbb R}\cos(tu)\phi_f(u)\, du,
$$ 
$$
K(t)=\sum_{u\in\mathbb R}\sin(tu)f_u+\int_{\mathbb R}\sin(tu)\phi_f(u)\, du.
$$ 
Then $F^c(it)=\overline{H(t)}+i\overline{K(t)}$. If $K(t)=0$, then $F^c(it)=\overline{F(it)}$. Otherwise, then for $\tilde i=-\overline{K(t)}i\overline{K(t)}^{-1}$ we get 
$$
F^c(\tilde it)=\overline{F(it)}=\overline{H(t)}-\overline{K(t)}i\overline{K(t)}^{-1}\overline{K(t)}=\overline{H(t)}-\overline{K(t)}i=\overline{F(it)}.
$$
This shows that $\overline{F(\mathbb S\mathbb R)}\subseteq F^c(\mathbb S\mathbb R)$, and the reverse follows by substituting $F^c$ instead of $F$ using the fact that $(F^c)^c=F$. So in particular, 
$$
\inf_{p\in\mathbb S\mathbb R}|F^c(p)|=\inf_{p\in\mathbb S\mathbb R}|F(p)|. 
$$
Going back to (2.8), we get
$$
\inf_{i,j,t}|\det\omega_{i,j}(F)(t)|\geq (\inf_{p\in\mathbb S\mathbb R}|F(p)|)^2>0.
$$
$(iv) \implies (v)$: 
$$
0<\inf_{i,j,t}|\det\omega_{i,j}(F)(t)|=\inf_{i,t}|F(it)||F^c(i_tt)|\leq (\inf_{p\in\mathbb S\mathbb R}|F(p)|)(\sup_{p\in\mathbb S\mathbb R}|F^c(p)|).
$$
Since $F^c$ is bounded, we can divide by $\sup_{p\in\mathbb S\mathbb R}|F^c(p)|$ and get
$$
0< \inf_{p\in\mathbb S\mathbb R}|F(p)|.
$$
$(v) \implies (vi)$: Trivial. \\
$(vi) \implies (v)$: Assume by negation that $\inf_{p\in\mathbb S\mathbb R}|F(p)|=0$. Then there exist $(i_n)_{n\in\mathbb N}\subset\mathbb S$, $(t_n)_{n\in\mathbb N}\subset\mathbb R$ such that  $\lim_{n\to\infty}F(i_nt_n)=0$. Since $\mathbb S$ is compact, we can assume without loss of generality that $(i_n)_{n\in\mathbb N}$ converges to $i_0\in\mathbb S$. Observe that for all $i,j\in\mathbb S$ and $t\in\mathbb R$ we have
$$
|F(it)-F(jt)|=|\sum_{u\in\mathbb R}(i-j)\sin(tu)f_u+\int_{\mathbb R}(i-j)\sin(tu)\phi_f(u)\, du|\leq\|F\||i-j|.
$$
This implies that $\lim_{n\to\infty}F(i_0t_n)=0$, which contradicts $\inf_{t\in\mathbb R}|F(i_0t)|>0$.
\end{proof}

\begin{remark}
As is shown in \cite{css} for slice hyperholomorphic functions, it is easy to derive the formula 
$$
F^{-\star}=\dfrac{F^c}{F\star F^c}.
$$ 
Not only is $F\star F^c$ invertible if and only if $F$ is, but also all of its coefficients (in the almost periodic and continuous components) are real. Thus, $\star$-inversion is the same as pointwise inversion.
\end{remark}

\begin{corollary}
Invertibility in $\mathcal APW(\mathbb S\mathbb R,\mathbb H)$  is equivalent to conditions (ii)-(vi) for functions in the aforementioned subalgebra.
\end{corollary}

\begin{proof}
We only need to show why $\mathcal APW(\mathbb S\mathbb R,\mathbb H)$ is closed with respect to inversion. Let $F\in\mathcal APW(\mathbb S\mathbb R,\mathbb H)$ be invertible in $\mathcal B(\mathbb S\mathbb R,\mathbb H)$. Writing 
$$
F(it)=\sum_{u\in\mathbb R}e^{itu}f_u,
$$
$$
F^{-\star}(it)=\sum_{u\in\mathbb R}e^{itu}g_u+\int_{\mathbb R}e^{itu}\phi_g(u)du,
$$ 
we get
$$
1=(F\star G)(it)=\sum_{u\in\mathbb R}e^{itu}\sum_{v\in\mathbb R}f_vg_{u-v}+\int_{\mathbb R}e^{itu}du\sum_{v\in\mathbb R}f_v\phi_g(u-v).
$$
This implies that 
$$
1=\sum_{u\in\mathbb R}e^{itu}\sum_{v\in\mathbb R}f_vg_{u-v}.
$$ 
So $G'(it)=\sum_{u\in\mathbb R}e^{itu}g_u$ also satisfies $1=(F\star G')(it)$, and since the inverse is unique, $F^{-\star}=G'\in\mathcal APW(\mathbb S\mathbb R,\mathbb H)$.
\end{proof}

We now extend the previous definitions to the case of quaternionic matrix-valued functions. To clarify, the rest of this section is mostly a rewriting of the preceding results, with the real difference being the proof of Theorem 2.18. \\
First, let $\|\cdot\|_n$ on $\mathbb H^{n\times n}$ denote the operator norm (with respect to the Euclidean norm on $\mathbb H^{n\times 1}$). We will start with the discrete Wiener case:

\begin{e-definition}
We denote by $\mathcal W_{\mathbb H}^{n\times n}$ the set of matrix-valued functions of the form
\begin{equation}\label{effe}
F(it)=\sum_{u\in\mathbb Z}p^uF_u
\end{equation}
where the $F_u\in\mathbb H^{n\times n}$ satisfy
\[
\sum_{u\in\mathbb Z}\|F_u\|_n<\infty.
\]
This set can be endowed with the multiplication
$$
(F\star G)(it)=\sum_{u\in\mathbb Z}p^u\sum_{v\in\mathbb Z}F_vG_{u-v}.
$$
We also define
$$
\| F\|= \sum_{u\in\mathbb Z}\|F_u\|_n.
$$
\end{e-definition}

\begin{e-definition}
We denote by $\mathcal APW(\mathbb S\mathbb R,\mathbb H^{n\times n})$ the set of matrix-valued functions of the form
\begin{equation}\label{effe}
F(it)=\sum_{u\in\mathbb R}e^{itu}F_u
\end{equation}
where the $F_u\in\mathbb H^{n\times n}$ are quaternionic matrices vanishing for all but a countable subset of $u \in \mathbb{R}$, and
\[
\sum_{u\in\mathbb R}\| F_u\|_n<\infty.
\]
This set can be endowed with the multiplication
$$
(F\star G)(it)=\sum_{u\in\mathbb R}e^{itu}\sum_{v\in\mathbb R}F_vG_{u-v}.
$$
We also define
$$
\| F\|= \sum_{u\in\mathbb R}\| F_u\|_n.
$$
\end{e-definition}

\begin{remark}
$\mathcal W_{\mathbb H}^{n\times n}$ can be identified with a subalgebra of $\mathcal APW(\mathbb S\mathbb R,\mathbb H^{n\times n})$ by substituting $p=e^{it}$.
\end{remark}

\begin{e-definition}
We denote by $\mathcal W(\mathbb S\mathbb R,\mathbb H^{n\times n})$ the set of all matrix-valued functions of the form
\begin{equation}\label{effe}
F(it)=C_F+\int_{\mathbb R}e^{itu}\Phi_F(u)\, du
\end{equation}
where $\Phi_F\in L_1(\mathbb R,\mathbb H^{n\times n})$ and $C_F\in\mathbb H^{n\times n}$ is a constant matrix.
This set can be endowed with the multiplication
\begin{equation*}
\begin{split}
(F\star G)(it)&=C_FC_G+\int_{\mathbb R}e^{itu}C_F\Phi_G(u)du+\int_{\mathbb R}e^{itu}\Phi_F(u)C_Gdu\\
&+\int_{\mathbb R}e^{itu} (\Phi_F\circ \Phi_G)(u)\ du.
\end{split}
\end{equation*}
We also define
$$
\| F\|= \| C_F\|_n+\int_{\mathbb R}\| \phi_F\|_n du.
$$
\end{e-definition}

\begin{e-definition}
We denote by $\mathcal B(\mathbb S\mathbb R,\mathbb H^{n\times n})$ the sum of $\mathcal APW(\mathbb S\mathbb R,\mathbb H^{n\times n})$ and $\mathcal W(\mathbb S\mathbb R,\mathbb H^{n\times n})$, which is to say the set of all matrix-valued functions of the form
\begin{equation}\label{effe}
F(it)=\sum_{u\in\mathbb R}e^{itu}F_u+\int_{\mathbb R}e^{itu}\Phi_F(u)\, du
\end{equation}
where $\Phi_F\in L_1(\mathbb R,\mathbb H^{n\times n})$ and the $F_u\in\mathbb H^{n\times n}$ satisfy
\[
\sum_{u\in\mathbb R}\| F_u\|_n<\infty.
\]
This set can be endowed with the multiplication
\begin{equation}
\begin{split}
(F\star G)(it)&=\sum_{u\in\mathbb R}e^{itu}\sum_{v\in\mathbb R}F_vG_{u-v}+\int_{\mathbb R}e^{itu}du\sum_{v\in\mathbb R}F_v\Phi_G(u-v) \\
&+\int_{\mathbb R}e^{itu}du\sum_{v\in\mathbb R}\Phi_F(u-v)G_v+\int_{\mathbb R}e^{itu} (\Phi_F\circ \Phi_G)(u)\, du.
\end{split}
\end{equation}
We also define
$$
\| F\|= \sum_{u\in\mathbb R}\| F_u\|_n+\int_{\mathbb R}\| \phi_F\|_n du.
$$
\end{e-definition}

\begin{proposition}
$\mathcal B(\mathbb S\mathbb R,\mathbb H^{n\times n})$ endowed with the $\star$-multiplication is a real Banach algebra, which contains $\mathcal APW(\mathbb S\mathbb R,\mathbb H)$, $\mathcal W(\mathbb S\mathbb R,\mathbb H)$ and $\mathcal W_{\mathbb H}^{n\times n}$ as closed subalgebras.
\end{proposition}

\begin{proof}
That $\mathcal B(\mathbb S\mathbb R,\mathbb H^{n\times n})$ is a real Banach space follows from Proposition 2.3, since convergence is $\mathcal B(\mathbb S\mathbb R,\mathbb H^{n\times n})$ equivalent to entry-wise convergence in $\mathcal B(\mathbb S\mathbb R,\mathbb H)$. Similarly, $\mathcal APW(\mathbb S\mathbb R,\mathbb H^{n\times n})$ and $\mathcal W(\mathbb S\mathbb R,\mathbb H^{n\times n})$ and $\mathcal W_{\mathbb H}^{n\times n}$ are closed subalgebras. Finally:
\begin{equation*}
\begin{split}
\|F\star G\|&=\sum_{u\in\mathbb R}\|\sum_{v\in\mathbb R}F_vG_{u-v}\|_n+\int_{\mathbb R}\|\sum_{v\in\mathbb R}F_v\phi_G(u-v)+\sum_{v\in\mathbb R}\phi_F(u-v)G_v\\
&+(\phi_F\circ \phi_G)(u)\|_ndu+\leq\sum_{u\in\mathbb R}\sum_{v\in\mathbb R}\|F_v\|_n\|G_{u-v}\|_n\\
&+\int_{\mathbb R}\sum_{v\in\mathbb R}\|F_v\|_n\|\phi_G(u-v)\|_n+\sum_{v\in\mathbb R}\|\phi_F(u-v)\|_n\|G_v\|_n\\
&+(\|\phi_F\|_n\circ \|\phi_G\|_n)(u)du=\|F\|\|G\|.
\end{split}
\end{equation*}
\end{proof}

We now extend the definition of the map $\omega=\omega_{i,j}$ :

\begin{e-definition}
Let $F\in \mathcal B(\mathbb S\mathbb R,\mathbb H^{n\times n})$ be given by 
$$
F(it)=\sum_{u\in\mathbb R}e^{itu}F_u+\int_{\mathbb R}e^{itu}\Phi_F(u)\, du.
$$ 
Then
\[
\omega(F)(t)=\sum_{u\in\mathbb R}e^{itu}\chi(F_u)+\int_{\mathbb R}e^{itu}\chi(\Phi_F(u)) du,
\]
where $\chi$ is defined as in the introduction (for matrices). Similarly, letting $F\in \mathcal W_{\mathbb H}^{n\times n}$ be given by $F(p)=\sum_{u\in\mathbb Z}p^uF_u$, we have
\[
\omega(F)(z)=\sum_{u\in\mathbb Z}z^u\chi(F_u).
\]
\end{e-definition}
It is again immediate that $\omega$ maps $\mathcal B(\mathbb S\mathbb R,\mathbb H^{n\times n})$ injectively into $\mathcal B^{2n\times 2n}(\mathbb R,\mathbb C_i)$ with values in $\mathbb C_i^{2n\times 2n}$. Then, it maps $\mathcal APW(\mathbb S\mathbb R,\mathbb H^{n\times n})$ injectively into \newline $\mathcal APW^{2n\times 2n}(\mathbb R,\mathbb C_i)$, and $\mathcal W_{\mathbb H}^{n\times n}$ injectively into $\mathcal W^{2n\times 2n}$.

\begin{lemma}\label{lemmastar}
Let $F,G\in\mathcal B(\mathbb S\mathbb R, \mathbb H^{n\times n})$. Then
$$
\omega(F\star G)(t)=\omega(F)(t)\omega(G)(t), \quad t \in \mathbb{R}.
$$
\end{lemma}

\begin{proof}
The proof of Lemma 2.5 still applies.
\end{proof}

\begin{theorem}
Let $F\in\mathcal B(\mathbb S\mathbb R,\mathbb H^{n\times n})$. The following are equivalent{\rm :}
\begin{enumerate}
\item[(i)] The function $F$ is invertible in $\mathcal B(\mathbb S\mathbb R,\mathbb H^{n\times n})$.
\item[(ii)] There exist orthogonal $i,j\in\mathbb S$ such that $$\inf_{t\in\mathbb R}  |\det \omega_{i,j}(F)(t)|>0.$$
\item[(iii)] For any orthogonal $i,j\in\mathbb S,$ $$\inf_{t\in\mathbb R}  |\det \omega_{i,j}(F)(t)|>0.$$
\item[(iv)] $$\inf_{i\perp j\in\mathbb S, t\in\mathbb R}  |\det \omega_{i,j}(F)(t)|>0.$$
\end{enumerate}
\end{theorem}

\begin{proof}
$(i) \implies (iv)$:
By the multiplicativity of $\omega$ we have that 
\begin{equation}
|\det \omega_{i,j}(F)(t)||\det \omega_{i,j}(F^{-\star})(t)|=1.
\end{equation}
Note that each entry of $\omega_{i,j}(G)(t)$ is an entry of either $A(it), B(it), \overline{A(-it)},$ or $-\overline{B(-it)}$, where $G(it)=A(it)+B(it)j$. For any $1\leq m,k\leq 2n$ we have
$$
\sup_{i,t}\{|a_{m,k}(it)|,|b_{m,k}(it)|\}\leq \sup_{i,t}|a_{m,k}(it)+b_{m,k}(it)j|\leq \|G\|,
$$
since the modulus of any entry of $G(it)$ is at most $\|G(it)\|_n$, which in turn is bounded by $\|G\|$ for all $i\in\mathbb S,t\in\mathbb R$.
The determinant of a $2n\times 2n$ matrix is (up to a sign for each term) a sum of $(2n)!$ products of $2n$ entries, so by the triangle inequality
$$
\sup_{i,j,t} |\det \omega_{i,j}(G)(t)|\leq(2n)!\|G\|^{2n}.
$$
Thus it follows from (2.14) that
$$
\inf_{i,j,t}|\det \omega_{i,j}(F)(t)|\geq \dfrac{1}{(2n)!\|F^{-\star}\|^{2n}}>0.
$$
$(iv) \implies (iii) \implies (ii)$: Trivial.\\
$(ii) \implies (i)$: Let us assume that $$\inf_{t\in\mathbb R} |\det \omega_{i,j}(F)(t)|>0$$ for some orthogonal $i,j\in\mathbb S$. As a scalar complex function in the algebra $\mathcal B(\mathbb R,\mathbb C)$, $d(it):=\det \omega_{i,j}(F)(t)$ is invertible. Then we see that
\begin{equation*}
{\omega_{i,j}(F)}^{-1}(t)=\dfrac{1}{d(it)}\adj(\omega_{i,j}(F)(t))
\end{equation*}
is a matrix function in the algebra $\mathcal B(\mathbb R,\mathbb C^{2n \times 2n})$, since each entry of the adjugate matrix is a product of functions in $\mathcal B(\mathbb R,\mathbb C)$. The inverse matrix corresponds to a $G\in\mathcal B(\mathbb S\mathbb R,\mathbb H^{n\times n})$, in the sense that $\omega_{i,j}(G)(t)={\omega_{i,j}(F)}^{-1}(t)$, if and only if
\begin{equation}
J_n\overline{\omega_{i,j}(F)^{-1}(-t)}J_n^T={\omega_{i,j}(F)}^{-1}(t),
\end{equation}
where 
$$
J_n=\left[\begin{matrix} 0 &I_n \\
-I_n &0
\end{matrix}\right].
$$
To show (2.14), we simply observe that
$$
J_n\overline{\omega_{i,j}(F)(-t)}J_n^T={\omega_{i,j}(F)}(t),
$$
and invert both sides (noting that $J_n^T=J_n^{-1}$). Thus, $G$ exists and $$\omega_{i,j}(F\star G)=I_{2n},$$ implying $G=F^{-1}$.
\end{proof}

\begin{corollary}
Invertibility in $\mathcal APW(\mathbb S\mathbb R,\mathbb H^{n\times n})$, $\mathcal W(\mathbb S\mathbb R,\mathbb H^{n\times n})$ is equivalent to conditions (ii)-(iv) for matrix-valued functions in the aforementioned subalgebras. Likewise in $\mathcal W_{\mathbb H}^{n\times n}$, with the following modifications:
\begin{enumerate}
\item[(ii)] There exist orthogonal $i,j\in\mathbb S$ such that $\det \omega_{i,j}(F)(z)\neq 0$ for all \newline $z\in\partial\mathbb B\cap \mathbb C_i$.
\item[(iii)] $\det \omega_{i,j}(F)(z)\neq 0$ for any orthogonal $i,j\in\mathbb S$ and any $z\in\partial\mathbb B\cap \mathbb C_i$.
\item[(iv)] $$\min_{i\perp j\in\mathbb S, z\in\partial\mathbb B\cap \mathbb C_i}  |\det \omega_{i,j}(F)(z)|>0.$$
\end{enumerate}
\end{corollary}

\begin{proof}
All the subalgebras are closed with respect to inversion, as can be shown similarly to the proof of Corollary 2.9.
\end{proof}

\section{Factorization and the Riemann-Hilbert problem}

We limit the discussion to the matrix-valued discrete and continuous algebras. First, let us reiterate definitions of subalgebras discussed in \cite{acks} in the scalar-valued case.
 
\begin{e-definition}
\hfill \break
\begin{enumerate}
\item[(i)] We denote by $\mathcal W_{\mathbb H,+}^{n\times n}$ (resp. $\mathcal W_{\mathbb H,-}^{n\times n}$) the set of elements 
$$
F(p)=\sum_{u\in\mathbb Z}
p^uF_u\,\in\,\mathcal W_{\mathbb H}^{n\times n}
$$ 
for which $F_u=0$ for $u<0$ (resp.  for $u>0$).
\item[(ii)] We denote by $\mathcal W_+(\mathbb S\mathbb R,\mathbb H^{n\times n})$ (resp. $\mathcal W_-(\mathbb S\mathbb R,\mathbb H^{n\times n})$) the set of elements 
$$
F(it)=C_F+\int_{\mathbb R}e^{itu}\Phi_F(u)\, du
$$
for which $\Phi_F(u)=0$ for $u<0$ (resp.  for $u>0$).
\end{enumerate}
\end{e-definition}

Note that any function in $\mathcal W_{\mathbb H,+}^{n\times n}$ is slice hyperholomorphic on the open unit ball $\mathbb B:=\{p\in\mathbb H : |p|<1\}$, while any function in $\mathcal W_{\mathbb H,-}^{n\times n}$ is slice hyperholomorphic on $\mathbb H\setminus\overline{\rm \mathbb B}=\{p\in\mathbb H : |p|>1\}$. Likewise, any function in $\mathcal W_+(\mathbb S\mathbb R,\mathbb H^{n\times n})$ has a slice hyperholomorphic continuation to the left half-space $\{p\in\mathbb H : \Rea(p)<0\}$ (obtained by setting $p$ instead of $it$), while any function $\mathcal W_-(\mathbb S\mathbb R,\mathbb H^{n\times n})$ has a slice hyperholomorphic continuation to the right half-space $\{p\in\mathbb H : \Rea(p)>0\}$.

\begin{e-definition}
\hfill \break
\begin{enumerate}
\item[(i)] Given $F\in\mathcal W_{\mathbb H}^{n\times n}$, we say that $F$ admits a factorization if there exist a diagonal matrix $D(p)=\diag\left[p^{k_1},\dots,p^{k_n}\right]$ (where $k_1\geq k_2\geq\dots\geq k_n$ are integers) and invertible (with respect to their respective subalgebras) $F_\pm\in\mathcal W_{\mathbb H,\pm}^{n\times n}$ such that 
$$
F(p)=(F_-\star D\star F_+)(p).
$$
\item[(ii)] Given $F\in\mathcal W(\mathbb S\mathbb R,\mathbb H^{n\times n})$, we say that $F$ admits a factorization if there exist a diagonal matrix 
$$
D(p)=\diag\left[\left(\dfrac{p+1}{p-1}\right)^{k_1},\dots,\left(\dfrac{p+1}{p-1}\right)^{k_n}\right]
$$ 
(where $k_1\geq k_2\geq\dots\geq k_n$ are integers) and invertible (with respect to their respective subalgebras) $F_\pm\in\mathcal W_\pm(\mathbb S\mathbb R,\mathbb H^{n\times n})$ and such that 
$$
F(p)=(F_-\star D\star F_+)(p).
$$
\end{enumerate}
In both cases $k_1\geq k_2\geq\dots\geq k_n$ are called the factorization indices, and the factorization is called canonical if $D=I$.
\end{e-definition}

It is well known (see \cite{cg}) that invertible functions in the complex-valued counterparts of the aforementioned algebras, admit factorization. To be clear, the diagonal matrix in $\mathcal W(\mathbb R,\mathbb C^{n\times n})$ is of the form 
$$
D(t)=\diag\left[\left(\dfrac{t-i}{t+i}\right)^{k_1},\dots,\left(\dfrac{t-i}{t+i}\right)^{k_n}\right],
$$ 
and in the quaternionic case we set $p=it$. It is also known that in the complex case the factorization indices are unique, which is to say that the diagonal elements are uniquely determined by the function once the order along the diagonal is fixed (multiplication by the same elementary matrices from the left and right can alter the order). In order to establish these facts in the quaternionic setting, we need to first recall the complex Riemann-Hilbert problem and its connection to factorization. For our purposes it suffices to consider the problem in the context of the Wiener algebras, but it should be noted that there is a more general theory (studied in \cite{cg}).

\begin{e-definition}
\hfill \break
\begin{enumerate}
\item[(i)] Let $F\in\mathcal W_{\mathbb C}^{n\times n}$. The associated barrier problem is to describe all piecewise holomorphic vector functions  $\Psi(z)$ given by
\begin{equation}
  \Psi(z)=\begin{cases}
    \Psi_{+}(z), & \text{if $|z|<1$}.\\
    \Psi_{-}(z), & \text{if $|z|>1$}.
  \end{cases}
\end{equation}
and satisfying
\begin{equation}
\forall z\in\partial \mathbb D \quad \Psi_-(z)=F(z)\Psi_+(z),
\end{equation}
where 
$$
\Psi_{+}\in\mathcal W_{\mathbb C,+}, \quad \Psi_{-}(z)-r(z)\in\mathcal W_{\mathbb C,-}
$$ and $r$ is a polynomial. We define $\ord(\Psi)$ to be the (possibly negative) order of $\infty$ as a pole of $\Psi$.
\item[(ii)] Let $F\in\mathcal W(\mathbb R,\mathbb C^{n\times n})$. The associated barrier problem is to describe all piecewise holomorphic vector functions  $\Psi(t)$ given by
\begin{equation}
  \Psi(t)=\begin{cases}
    \Psi_{+}(t), & \text{if $\Ima(t)>0$}.\\
    \Psi_{-}(t), & \text{if $\Ima(t)<0$}.
  \end{cases}
\end{equation}
and satisfying
\begin{equation}
\forall t\in\mathbb R \quad \Psi_-(t)=F(t)\Psi_+(t),
\end{equation}
where 
$$\Psi_{+}\in\mathcal W_+(\mathbb R,\mathbb C^{n\times n}), \quad \Psi_{-}(t)-r\left(\dfrac{t-i}{t+i}\right)\in\mathcal W_-(\mathbb R,\mathbb C^{n\times n})
$$ 
and $r$ is a polynomial. We define $\ord(\Psi)$ to be the (possibly negative) order of $-i$ as a pole of $\Psi$.
\end{enumerate}
\end{e-definition}

\newpage
\begin{e-definition}
\hfill \break
\begin{enumerate}
\item[(i)] A solution set $\{\Psi_1(z),\dots,\Psi_s(z)\}$ of the barrier problem (3.2) is called complete if every solution $\Psi$ has a representation of the form 
$$\Psi(z)=\sum_{1\leq m\leq s}q_m(z)\Psi_m(z),
$$ 
where the $q_m$ are polynomials.
\item[(ii)] A solution set $\{\Psi_1(t),\dots,\Psi_s(t)\}$ of the barrier problem (3.4) is called complete if every solution has a representation of the form 
$$
\Psi(t)=\sum_{1\leq m\leq s}q_m\left(\dfrac{t-i}{t+i}\right)\Psi_m(t),
$$
where the $q_m$ are polynomials.
\end{enumerate}
\end{e-definition}

\begin{e-definition}
A solution set $\{\Psi_1,\dots,\Psi_n\}$ of eq. (3.2) (eq. (3.4)) is called admissible if $\Psi_1(0),\dots,\Psi_n(0)$ ($\Psi_1(i),\dots,\Psi_n(i)$) are linearly independent over $\mathbb C$, and 
$$
k_1:=\ord(\Psi_1)\geq k_2:=\ord(\Psi_2)\geq \dots\geq  k_n:= \ord(\Psi_n).
$$ 
$[k_1,\dots,k_n]$ is called the index set of $\Psi$.
\end{e-definition}

\begin{proposition}
Let $F\in\mathcal W_{\mathbb C}^{n\times n}$ ($F\in\mathcal W(\mathbb R,\mathbb C^{n\times n})$). Then there is an integer $\alpha(F)$ such that any non-zero solution $\Psi$ of eq. (3.2) (eq. (3.4)) satisfies $\ord(\Psi)\geq \alpha(F)$.
\end{proposition}

\begin{e-definition}
An admissible solution set $\{\Psi_1,\dots,\Psi_n\}$ of eq. (3.2) (eq. (3.4)) is called standard if its index set is minimal (with respect to the lexicographic order from left to right) among all admissible solution sets.
\end{e-definition}

\begin{theorem}
Let $F\in\mathcal W_{\mathbb C}^{n\times n}$ ($F\in\mathcal W(\mathbb R,\mathbb C^{n\times n})$). Given any standard solution set $\{\Psi_1,\dots,\Psi_n\}$ of the barrier problem (3.2) (problem (3.4)) we can obtain a factorization $F=F_-DF_+$ by setting 
$$
F_-(z)=[z^{-\ord(\Psi_1)}\Psi_{1,-}(z),\dots,z^{-\ord(\Psi_n)}\Psi_{n,-}(z)]
$$
$$
\left(F_-(t)=\left[\left(\dfrac{t-i}{t+i}\right)^{-\ord(\Psi_1)}\Psi_{1,-}(t),\dots,\left(\dfrac{t-i}{t+i}\right)^{-\ord(\Psi_n)}\Psi_{n,-}(t)\right]\right),
$$
$$
D(z)=\diag[z^{\ord(\Psi_1)},\dots,z^{\ord(\Psi_n)}]
$$ 
$$
\left(D(t)=\diag\left[\left(\dfrac{t-i}{t+i}\right)^{\ord(\Psi_1)},\dots,\left(\dfrac{t-i}{t+i}\right)^{\ord(\Psi_n)}\right]\right),
$$ 
$$
F_+=[\Psi_{1,+},\dots,\Psi_{n,+}]^{-1}.
$$ 
Conversely, given a factorization $F=F_-DF_+$, then the set  $\Psi_1,\dots,\Psi_n$ given by 
$$
[\Psi_{1,+},\dots,\Psi_{n,+}]=F_+^{-1}, \quad [\Psi_{1,-},\dots,\Psi_{n,-}]=F_- D
$$ 
is a complete and standard solution set.
\end{theorem}

\begin{corollary}
Any standard solution set of eq. (3.2) (eq. (3.4)) is complete.
\end{corollary}

Note that Theorem 3.8 remains valid if we choose a different order of indices along the diagonal $D$, as long as we use the respective order for the indices of the solution set. In this case, a standard solution set is defined with respect to the corresponding lexicographic order. By $s_1 \rightarrow s_2 \rightarrow s_3 \rightarrow \dots \rightarrow s_n$ we refer to the lexicographic order first determined by comparing the $s_1$-th entries of two index sets, then their $s_2$-th entries, and so on until the entries are not equal.

Rather than immediately extending these results to the quaternionic setting, we use them for the proof of the following theorem regarding the existence of factorization:

\begin{theorem}
Let $F\in\mathcal W_{\mathbb H}^{n\times n}$ ($F\in\mathcal W(\mathbb S\mathbb R,\mathbb H^{n\times n})$) be invertible. Then it admits a factorization and $D$ is unique.  If $n=1$, then $F_+,F_-$ are unique up to the transformation $c\star F_+,F_-c^{-1 }$, where $c\in\mathbb H$ is a non-zero constant.
\end{theorem}

\begin{proof}
Fix $i \perp j \in \mathbb S$. We will prove uniqueness first: If a factorization exists with an index set $k_1\geq k_2\geq\dots\geq k_n$ , then it induces a factorization 
$$
\omega(F)=\omega(F_-)\omega(D)\omega(F_+),
$$
where
$$
\omega(D)(z)=\begin{bmatrix}
    D(z) &0 \\
    0 &D(z)
\end{bmatrix} \quad
\left(\omega(D)(t)=\begin{bmatrix}
    D(it) &0 \\
    0 &D(it)
\end{bmatrix}\right)
$$
is a diagonal matrix in $\mathcal W^{n\times n}$ ($\mathcal W(\mathbb R,\mathbb C^{n\times n})$). It is known that (up to the order of indices), the diagonal element of a factorizable matrix function in $\mathcal W^{n\times n}$ ($\mathcal W(\mathbb R,\mathbb C^{n\times n})$) is unique. Thus $D$ is unique. If $n=1$, then 
$$
D(p)=p^k \left(D(p)=\left(\dfrac{p+1}{p-1}\right)^k\right)
$$
$\star$-commutes with any matrix (note that 
$$\dfrac{it+1}{it-1}=1-2\int_{0}^{\infty}e^{itu}e^{-u}du,
$$ 
so the components are real). Thus, if we have two factorizations
$$
F_{-}\star p^k \star F_{+}=G_{-}\star p^k \star G_{+} \quad \left(F_{-}\star \left(\dfrac{p+1}{p-1}\right)^k \star F_{+}=G_{-}\star \left(\dfrac{p+1}{p-1}\right)^k \star G_{+}\right),
$$
then
$$
G_{-}^{-\star}\star F_{-}=G_{+}\star F_{+}^{-\star}
$$
must be a constant $c \in \mathbb H$ since it is in 
$$
\mathcal W_{\mathbb H,+}^{n\times n} \cap \mathcal W_{\mathbb H,-}^{n\times n} \quad (\mathcal W_+(\mathbb S\mathbb R,\mathbb H^{n\times n}) \cap \mathcal W_-(\mathbb S\mathbb R,\mathbb H^{n\times n}),
$$ 
meaning that it has a bounded, slice hyperholomorphic continuation to $\mathbb H$ (rendering it constant by the quaternionic version of Liouville's theorem). $c\neq 0$ since $G_{-}^{-\star}\star F_{-}$ must be invertible. So $G_+=c\star F_{+}$ and $G_-=F_-c^{-1}$, and it is obvious that any $\neq 0$ satisfies the factorization.\par
Existence: Denote $F'(z)=\omega(F)(z)$ ($F'(t)=\omega(F)(t)$). It admits a factorization $F'=F'_-D'F'_+$, where the indices are ordered as follows: 
$$
k_1\geq k_{n+1}\geq k_2\geq k_{n+2}\geq \dots \geq k_n \geq k_{2n}.
$$ 
Let us consider the set of piecewise holomorphic functions $\Psi_1,\dots,\Psi_{2n}:\mathbb C_i \rightarrow \mathbb C_i^{2n}$ given by 
$$
[\Psi_{1,+},\dots,\Psi_{2n,+}]={F'_+}^{-1}, [\Psi_{1,-},\dots,\Psi_{2n,-}]=F'_-D'.
$$ 
By Theorem 3.8, it is a complete and standard solution set for the barrier problem $\Psi_-=F'\Psi_+$ with respect to the lexicographic order 
$$
1 \rightarrow n+1 \rightarrow 2 \rightarrow n+2 \rightarrow \dots \rightarrow n \rightarrow 2n.
$$ 
Let us see that for all $1\leq m\leq n$, 
$$
\tilde{\Psi}_m(z):=J_n\overline{\Psi_m(\overline z)} \quad \left(\tilde{\Psi}_m(t):=J_n\overline{\Psi_m(-\overline t)}\right)
$$ 
is also a solution of the aforementioned barrier problem. Indeed:
\begin{equation}
\begin{split}
\forall z\in\partial\mathbb B\cap\mathbb C_i \quad F'(z)\tilde{\Psi}_{m,+}(z)&=J_n\overline{F'(\overline z)}J_n^T(J_n\overline{\Psi_{m,+}(\overline z)})\\
&=J_n\overline{F'(\overline z)}\overline{\Psi_{m,+}(\overline z)}=J_n\overline{\Psi_{m,-}(\overline z)}=\tilde{\Psi}_{m,-}(z).
\end{split}
\end{equation}
\begin{equation}
\begin{split}
\forall t\in\mathbb R \quad F'(t)\tilde{\Psi}_{m,+}(t)&=J_n\overline{F'(-t)}J_n^T(J_n\overline{\Psi_{m,+}(-t)})\\
&=J_n\overline{F'(-t)}\overline{\Psi_{m,+}(-t)}=J_n\overline{\Psi_{m,-}(-t)}=\tilde{\Psi}_{m,-}(t).
\end{split}
\end{equation}
Moreover, $\ord(\tilde{\Psi}_m)=\ord(\Psi_m)$ for every $1\leq m\leq n$, since $f(z) \mapsto \overline{f(\overline z)}$ preserves the order at $\infty$ ($f(t) \mapsto \overline{f(-\overline t)}$ preserves the order at $-i$). We will soon show that there exist $1\leq m_1,\dots,m_n\leq 2n$ such that 
$$
\{\tilde{\Psi}_{m_1},\dots,\tilde{\Psi}_{m_n},\Psi_{m_1},\dots,\Psi_{m_n}\}
$$ 
is a standard solution set (with respect to the aforementioned lexicographic order), leading to a new factorization $F'=\tilde{F'_-}D'\tilde{F'_+}$ given by
$$
[\tilde{\Psi}_{m_1,+},\dots,\tilde{\Psi}_{m_n,+},\Psi_{m_1,+},\dots,\Psi_{m_n,+}]=\tilde{F'}_+^{-1},
$$
$$
[\tilde{\Psi}_{m_1,-},\dots,\tilde{\Psi}_{m_n,-},\Psi_{m_1,-},\dots,\Psi_{m_n,-}]=\tilde{F'_-}D'.
$$
This implies that $\tilde{F'}_+^{-1}(z)$ ($z=it$) is of the form
$$
\left[
\begin{matrix}
A(z) &B(z) \\
-\overline{B(\overline z)} &\overline{A(\overline z)}
\end{matrix}
\right],
$$
and likewise for $\tilde{F'_-}D'$, while $D'$ is a diagonal matrix with indices $k_{m_1}\geq \dots \geq k_{m_n}$ along the first half of the diagonal, and the same indices along the second half (since $\ord(\tilde{\Psi}_m)=\ord(\Psi_m)$), which further implies that $\tilde{F'_-}$ is also of the above form. Thus there exist $\tilde{F_-} \in  \mathcal W_{\mathbb H,-}^{n\times n}$ ($\mathcal W_-(\mathbb S\mathbb R,\mathbb H^{n\times n})$), $\tilde{F_+} \in  \mathcal W_{\mathbb H,+}^{n\times n}$ ($\mathcal W_+(\mathbb S\mathbb R,\mathbb H^{n\times n})$) and a diagonal $D \in W_{\mathbb H}^{n\times n}$ ($\mathcal W(\mathbb S\mathbb R,\mathbb H^{n\times n})$), such that 
$$\omega(\tilde{F_-})=\tilde{F'_-},\quad \omega(\tilde{F_+})=\tilde{F'_+}, \quad \omega(D)=D'.
$$
Since $\omega(F)=\omega(\tilde{F'_+} \star D \star\tilde{ F'_+})$ it follows that $F=\tilde{F'_+} \star D \star\tilde{ F'_+}$. \par
It remains to prove the existence of the required $1\leq m_1,\dots,m_n\leq 2n$. To begin with, consider $\tilde{\Psi}_{2n}$. Since $\Psi_1,\dots,\Psi_{2n}$ form a complete and admissible set, they are linearly independent at $0$ (at $i$) and there exist (complex) polynomials $p_1,\dots,p_{2n}$ such that 
$$
\tilde{\Psi}_{2n}(z)=\sum_{s=1}^{2n}p_s(z)\Psi_s(z) \quad \left(\tilde{\Psi}_{2n}(t)=\sum_{s=1}^{2n}p_s\left(\dfrac{t-i}{t+i}\right)\Psi_s(t)\right).
$$
There exists $1\leq l_1 \leq 2n-1$ such that $\Psi_1,\dots,\Psi_{l_1-1},\tilde{\Psi}_{2n},\Psi_{l_1+1},\dots,\Psi_{2n}$ are linearly independent at $0$ (at $i$). Indeed, otherwise the vectors $\tilde{\Psi}_{2n}(0),\Psi_{2n}(0)$ ($\tilde{\Psi}_{2n}(i),\Psi_{2n}(i)$) must be linearly dependent. Since $\Psi_{2n}(0)$ ($\Psi_{2n}(i)$) is non-zero, there exists $\alpha \in \mathbb C_i$ such that
\begin{equation}
\tilde{\Psi}_{2n}(0)=\alpha\Psi_{2n}(0) \quad \left(\tilde{\Psi}_{2n}(i)=\alpha\Psi_{2n}(i)\right).
\end{equation}
Since by definition $\tilde{\Psi}_s(0)=J_n\overline{\Psi_s(0)}$ ($\tilde{\Psi}_s(i)=J_n\overline{\Psi_s(i)}$) for all $1\leq s\leq n$, we have 
$$
J_n\overline{\tilde{\Psi}_s(0)}=J_n^2\Psi_s(0)=-\Psi_s(0) \quad \left(J_n\overline{\tilde{\Psi}_s(i)}=-\Psi_s(i)\right).
$$ 
By conjugating equation (3.7) and then multiplying by $J_n$ from the left, we obtain
\begin{equation}
-\Psi_{2n}(0)=\overline\alpha\tilde{\Psi}_{2n}(0) \quad \left(-\Psi_{2n}(i)=\overline\alpha\tilde{\Psi}_{2n}(i)\right).
\end{equation}
Setting (3.7) in (3.8), we get $|\alpha|^2=-1$, which is a clear contradiction.
Now reorder $\{\Psi_1,\dots,\Psi_{l_1-1},\tilde{\Psi}_{2n},\Psi_{l_1+1},\dots,\Psi_{2n}\}$ to obtain a standard solution set 
$$
\{\Psi_{m_1},\dots,\Psi_{m_{n-1}},\\\tilde{\Psi}_{2n},\Psi_{m_{n+1}},\dots,\Psi_{m_{2n-1}},\Psi_{2n}\}.
$$ 
It is minimal since $\ord(\tilde{\Psi}_{2n})=\ord(\Psi_{2n})\leq \ord(\Psi_{l_1})$. To avoid obfuscation, we may assume $m_s=s$ for all $s\in \{1,\dots,n-1,n+1,\dots,2n\}$ without loss of generality. We proceed by induction as follows: Assume that 
$$
\{\Psi_{1},\dots,\Psi_{u-1},\tilde{\Psi}_{n+u},\dots,\tilde{\Psi}_{2n},\Psi_{n+1},\\\dots,\Psi_{2n}\}
$$ 
is a standard solution set. We want to show that 
$$
\{\Psi_{1},\dots,\Psi_{u-2},\tilde{\Psi}_{n+u-1},\\\dots,\tilde{\Psi}_{2n},\Psi_{n+1},\dots,\Psi_{2n}\}
$$
is also a standard solution set. By the completeness and linear independence of the system, it suffices to show that $\tilde{\Psi}_{n+u-1},\dots,\tilde{\Psi}_{2n},\Psi_{n+u-1},\dots,\Psi_{2n}$ are linearly independent at $0$ (at $i$), since this implies that that we may replace some $\Psi_{l_u}$ with $\tilde{\Psi}_{n+u-1}$, where $l_u \in \{1,\dots,u-1,n+1,\dots,n+u-2\}$ and without loss of generality $l_u=u-1$. By negation, let $\beta,c_1,\dots,c_{n-u+1},d_1,\dots,d_{n-u+1} \in\mathbb C_i$ be such that
\begin{equation}
\begin{split}
&\tilde{\Psi}_{n+u-1}(0)=\beta\Psi_{n+u-1}(0)+c_1\tilde{\Psi}_{n+u}(0)+\dots+c_{n-u+1}\tilde{\Psi}_{2n}(0)+d_1{\Psi}_{n+u}(0)\\
&+\dots+d_{n-u+1}{\Psi}_{2n}(0) \\
&\Bigg(\tilde{\Psi}_{n+u-1}(i)=\beta\Psi_{n+u-1}(i)+c_1\tilde{\Psi}_{n+u}(i)+\dots+c_{n-u+1}\tilde{\Psi}_{2n}(i)+d_1{\Psi}_{n+u}(i)\\
&+\dots+d_{n-u+1}{\Psi}_{2n}(i)\Bigg).
\end{split}
\end{equation}
By conjugating equation (3.9) and then multiplying by $J_n$ from the left, we obtain
\begin{equation}
\begin{split}
-\Psi_{n+u-1}(0)&=\overline\beta\tilde{\Psi}_{n+u-1}(0)+\overline{d_1}\tilde{\Psi}_{n+u}(0)+\dots+\overline{d_{n-u+1}}\tilde{\Psi}_{2n}(0)\\
&-\overline{c_1}{\Psi}_{n+u}(0)-\dots-\overline{c_{n-u+1}}{\Psi}_{2n}(0) \\
\Bigg(-\Psi_{n+u-1}(i)&=\overline\beta\tilde{\Psi}_{n+u-1}(i)+\overline{d_1}\tilde{\Psi}_{n+u}(i)+\dots+\overline{d_{n-u+1}}\tilde{\Psi}_{2n}(i)\\
&-\overline{c_1}{\Psi}_{n+u}(i)-\dots-\overline{c_{n-u+1}}{\Psi}_{2n}(i)\Bigg).
\end{split}
\end{equation}
Rewriting (3.10) and multiplying (3.9) by $-\overline\beta$ yields
\begin{equation}
\begin{split}
&\Psi_{n+u-1}(0)+\overline{d_1}\tilde{\Psi}_{n+u}(0)+\dots+\overline{d_{n-u+1}}\tilde{\Psi}_{2n}(0)-\overline{c_1}{\Psi}_{n+u}(0)-\dots\\
&-\overline{c_{n-u+1}}{\Psi}_{2n}(0)=-|\beta|^2\Psi_{n+u-1}(0)-\overline\beta(c_1\tilde{\Psi}_{n+u}(0)+\dots+c_{n-u+1}\tilde{\Psi}_{2n}(0)\\
&+d_1{\Psi}_{n+u}(0)+\dots+d_{n-u+1}{\Psi}_{2n}(0)) \\
&\Bigg(\Psi_{n+u-1}(i)+\overline{d_1}\tilde{\Psi}_{n+u}(i)+\dots+\overline{d_{n-u+1}}\tilde{\Psi}_{2n}(i)-\overline{c_1}{\Psi}_{n+u}(i)-\dots\\
&-\overline{c_{n-u+1}}{\Psi}_{2n}(i)=-|\beta|^2\Psi_{n+u-1}(i)-\overline\beta(c_1\tilde{\Psi}_{n+u}(i)+\dots+c_{n-u+1}\tilde{\Psi}_{2n}(i)\\
&+d_1{\Psi}_{n+u}(i)+\dots+d_{n-u+1}{\Psi}_{2n}(i))\Bigg).
\end{split}
\end{equation}
By the linear independence of $\tilde{\Psi}_{n+u},\dots,\tilde{\Psi}_{2n},\Psi_{n+u-1},\Psi_{n+u}\dots,\Psi_{2n}$ at $0$ (at $i$), we can equate coefficients and get $1=-|\beta|^2$, which is again a contradiction.

\end{proof}

\begin{remark}
It is known that invertibility is not sufficient for factorizability in $\mathcal APW(\mathbb R,\mathbb C^{n\times n})$ for $n\geq 2$ (although it is sufficient for $n=1$). This makes it problematic to go about a proof for the quaternionic scalar case, since it is not clear that for any invertible $F\in\mathcal APW(\mathbb S\mathbb R,\mathbb H)$ the $2\times 2$ matrix-valued function $\omega(F)$ is factorizable.
\end{remark}

We continue this section by exploring the connection between factorization and the Riemann-Hilbert problem in the quaternionic setting. Our aim is to extend Theorem 3.8 and the preceding definitions.

\begin{e-definition}
\hfill \break
\begin{enumerate}
\item[(i)] Let $F\in\mathcal W_{\mathbb H}^{n\times n}$. The associated barrier problem is to describe all piecewise slice hyperholomorphic vector functions  $\Phi(p)$ given by
\begin{equation}
  \Phi(p)=\begin{cases}
    \Phi_{+}(p), & \text{if $|p|< 1$}.\\
    \Phi_{-}(p), & \text{if $|p|>1$}.
  \end{cases}
\end{equation}
and satisfying
\begin{equation}
\forall p\in\partial \mathbb B \quad \Phi_-(p)=F(p)\star \Phi_+(p),
\end{equation}
where 
$$
\Phi_{+}\in\mathcal W_{\mathbb H,+}, \quad \Phi_{-}(p)-r(p)\in\mathcal W_{\mathbb H,-}
$$ 
and $r$ is a polynomial. We define $\ord(\Phi)$ to be the (possibly negative) order of $\infty$ as a pole of $\Phi$.
\item[(ii)] Let $F\in\mathcal W(\mathbb S\mathbb R,\mathbb H^{n\times n})$. The associated barrier problem is to describe all piecewise slice hyperholomorphic vector functions  $\Phi(p)$ given by
\begin{equation}
  \Phi(p)=\begin{cases}
    \Phi_{+}(p), & \text{if $\Rea(p)< 0$}.\\
    \Phi_{-}(p), & \text{if $\Rea(p)>0$}.
  \end{cases}
\end{equation}
and satisfying
\begin{equation}
\forall p\in\mathbb S \mathbb R \quad \Phi_-(p)=F(p)\star \Phi_+(p),
\end{equation}
where 
$$\Phi_{+}\in\mathcal W_+(\mathbb S\mathbb R,\mathbb H^{n\times n}), \quad \Phi_{-}(p)-r\left(\dfrac{p+1}{p-1}\right)\in\mathcal W_-(\mathbb S\mathbb R,\mathbb H^{n\times n})
$$ 
and $r$ is a polynomial. We define $\ord(\Phi)$ to be the (possibly negative) order of $1$ as a pole of $\Phi$.
\end{enumerate}
\end{e-definition}

\begin{e-definition}
\hfill \break
\begin{enumerate}
\item[(i)] A solution set $\{\Phi_1(p),\dots,\Phi_s(p)\}$ of the barrier problem (3.13) is called complete if every solution has a representation of the form 
$$
\sum_{1\leq k\leq s}\Phi_k(p)\star q_k(p),
$$ 
where the $q_k$ are polynomials.
\item[(ii)] A solution set $\{\Phi_1(p),\dots,\Phi_s(p)\}$ of the barrier problem (3.15) is called complete if every solution has a representation of the form 
$$
\sum_{1\leq k\leq s}\Phi_k(p)\star q_k\left(\dfrac{p+1}{p-1}\right),
$$ 
where the $q_k$ are polynomials.
\end{enumerate}
\end{e-definition}

\begin{e-definition}
A solution set $\{\Phi_1(p),\dots,\Phi_n(p)\}$ of eq. (3.13) (eq. (3.15)) is called admissible if $\Phi_1(0),\dots,\Phi_n(0)$ ($\Phi_1(-1),\dots,\Phi_n(-1)$) are linearly independent over $\mathbb H$ with respect to right multiplication, and 
$$
k_1=\ord(\Phi_1)\geq k_2=\ord(\Phi_2)\geq \dots\geq  k_n= \ord(\Phi_n).
$$
$[k_1,\dots,k_n]$ is called the index set of $\Psi$.
\end{e-definition}

\begin{proposition}
Let $F\in\mathcal W_{\mathbb H}^{n\times n}$ ($F\in\mathcal W(\mathbb S \mathbb R,\mathbb H^{n\times n})$). Then there is an integer $\alpha(F)$ such that any non-zero solution $\Psi$ of eq. (3.13) (eq. (3.15)) satisfies $\ord(\Psi)\geq \alpha(F)$.
\end{proposition}

\begin{proof}
Fix $i\perp j\in \mathbb S$. Any solution $\Psi$ of eq. (3.13) (eq. (3.15)) satisfies $\omega(\Psi_-)e_t=\omega(F)\omega(\Psi_+)e_t$ for $t=1,2$, where $e_1,e_2$ are the standard unit vectors in $\mathbb R^2$. For each $t=1,2$ we have that $\omega(\Psi)e_t$ is piecewise holomorphic as in the barrier problem (3.2) (barrier problem (3.4)). Moreover, $\max_{t=1,2}\ord(\omega(\Psi)e_t)=\ord(\Psi)$, since 
$$
\omega(p^kI_n)(z)=z^kI_{2n} \quad \left(\omega\left(\left(\dfrac{p+1}{p-1}\right)^kI_n\right)(t)=\left(\dfrac{t-i}{t+i}\right)^kI_{2n}\right)
$$ 
for all $k\in \mathbb Z$. By Proposition 3.6, there is an integer $\alpha(\omega(F))$ which is the minimal order of a solution of the barrier problem associated with $\omega(F)$. In particular, we get 
$$
\ord(\Psi)=\max_{t=1,2}\left(\ord(\omega(\Psi)e_t)\right)\geq \alpha(\omega(F)).
$$
\end{proof}

\begin{e-definition}
An admissible solution set $\{\Phi_1(p),\dots,\Phi_n(p)\}$ of eq. (3.13) (eq. (3.15)) is called standard if its index set is minimal (with respect to the lexicographic order from left to right) among all admissible solution sets.
\end{e-definition}

\begin{theorem}
Let $F\in\mathcal W_{\mathbb H}^{n\times n}$ ($F\in\mathcal W(\mathbb S \mathbb R,\mathbb H^{n\times n})$). Given any standard solution set $\{\Phi_1,\dots,\Phi_n\}$ of the barrier problem (3.13) (problem (3.15)) we can obtain a factorization $F(p)=(F_-\star D\star F_+)(p)$ by setting 
$$
F_-(p)=[p^{-\ord(\Phi_1)}\Phi_{1,-}(p),\dots,p^{-\ord(\Phi_n)}\Phi_{n,-}(p)]
$$
$$
\left(F_-(p)=[\left(\dfrac{p+1}{p-1}\right)^{-\ord(\Phi_1)}\star\Phi_{1,-}(p),\dots,\left(\dfrac{p+1}{p-1}\right)^{-\ord(\Phi_n)}\star\Phi_{n,-}(p)]\right),
$$
$$
D(p)=\diag\left[p^{\ord(\Phi_1)},\dots,p^{\ord(\Phi_n)}\right]
$$
$$
\left(D(p)=\diag\left[\left(\dfrac{p+1}{p-1}\right)^{\ord(\Phi_1)},\dots,\left(\dfrac{p+1}{p-1}\right)^{\ord(\Phi_n)}\right]\right),
$$
$$
F_+(p)=[\Phi_{1,+}(p),\dots,\Phi_{n,+}(p)]^{-\star}. 
$$
Conversely, given a factorization $F(p)=(F_-\star D\star F_+)(p)$, then the set $\{\Phi_1,\dots,\Phi_n\}$ given by 
\begin{equation*}
\begin{split}
[\Phi_{1,+}(p),\dots,\Phi_{n,+}(p)]&=F_+^{-\star}(p)\\
[\Phi_{1,-}(p),\dots,\Phi_{n,-}(p)]&=(F_-\star D)(p)
\end{split}
\end{equation*}
is a complete and standard solution set.
\end{theorem}

\begin{proof}
Fix $i\perp j\in \mathbb S$. In the first direction, for $1\leq k\leq n$ define 
$$
\Psi_k:=\omega(\Phi_k)e_1, \Psi_{n+k}:=\omega(\Phi_k)e_2.
$$ 
We want to show that the corresponding $F_-,D,F_+$ satisfy $F=F_-\star D\star F_+$. Note that 
$$
\omega(F_-)(z)=[z^{-\ord(\Psi_1)}\Psi_{1,-}(z),\dots,z^{-\ord(\Psi_{2n})}\Psi_{2n,-}(z)]
$$
$$
\left(\omega(F_-)(t)=\left[\left(\dfrac{t-i}{t+i}\right)^{-\ord(\Psi_1)}\Psi_{1,-}(t),\dots,\left(\dfrac{t-i}{t+i}\right)^{-\ord(\Psi_{2n})}\Psi_{2n,-}(t)\right]\right),
$$
$$
\omega(D)(z)=\diag\left[z^{\ord(\Psi_1)},\dots,z^{\ord(\Psi_{2n})}\right],
$$
$$
\left(\omega(D)(t)=\diag\left[(\dfrac{t-i}{t+i})^{\ord(\Psi_1)},\dots,(\dfrac{t-i}{t+i})^{\ord(\Psi_{2n})}\right]\right),
$$
$$
\omega(F_+)=[\Psi_{1,+},\dots,\Psi_{2n,+}]^{-\star}. 
$$
By Theorem 3.8, it suffices to show that $\{\Psi_1,\dots,\Psi_{2n}\}$ is a standard solution set for the barrier problem associated with $\omega(F)$, since then it will follow that 
$$
\omega(F)=\omega(F_-)\omega(D)\omega(F_+),
$$ 
implying $F=F_-\star D\star F_+$. That $\Psi_1,\dots,\Psi_{2n}$ are linearly independent over $\mathbb C_i$ at $0$ (at $i$) follows from the relation 
$$
\chi[\Phi_1(0),\dots,\Phi_n(0)]=\omega[\Phi_1,\dots,\Phi_n](0)=[\Psi_1(0),\dots,\Psi_{2n}(0)]
$$ 
$$
(\chi[\Phi_1(i),\dots,\Phi_n(i)]=\omega[\Phi_1,\dots,\Phi_n](i)=[\Psi_1(-1),\dots,\Psi_{2n}(-1)])
$$ 
and the fact that the columns of a matrix $A\in\mathbb H^{n\times n}$ are linearly independent over $\mathbb H$ (with respect to right multiplication) if and only if the columns of $\chi(A)$ are linearly independent over $\mathbb C_i$. To show minimality of the index set, let us recall that by the proof of Theorem 3.10, any standard solution set (with respect to the lexicographic order $1\rightarrow n+1\rightarrow \dots \rightarrow n \rightarrow 2n$) $\Gamma_1,\dots,\Gamma_{2n}$ for the barrier problem associated with $\omega(F)$ leads to another standard solution set of the form $\{\tilde{\Gamma}_{m_1},\dots,\tilde{\Gamma}_{m_n},{\Gamma}_{m_1},\dots,{\Gamma}_{m_n}\}$. Thus there exist $\Lambda_1,\dots,\Lambda_n$ such that $\omega(\Lambda_k)=[\tilde{\Gamma}_{m_k} \quad {\Gamma}_{m_k}]$ for all $1\leq k\leq n$, from which it follows that $\Lambda_1,\dots,\Lambda_n$ form an admissible solution set for the barrier problem associated with $F$. If the index set of $\Phi_1,\dots,\Phi_n$ is $[k_1,\dots,k_n]$ and that of $\Lambda_1,\dots,\Lambda_n$ is $[s_1,\dots,s_n]$, then by the minimality of $[k_1,\dots,k_n]$  we have $$[k_1,\dots,k_n]\leq [s_1,\dots,s_n].$$ But the index set of $\Psi_1,\dots,\Psi_{2n}$ is $[k_1,\dots,k_n,k_1,\dots,k_n]$, while that of $\tilde{\Gamma}_{m_1},\dots,\tilde{\Gamma}_{m_n},\newline{\Gamma}_{m_1},\dots,{\Gamma}_{m_n}$ is $[s_1,\dots,s_n,s_1,\dots,s_n]$. Then $[k_1,\dots,k_n]\geq [s_1,\dots,s_n]$ due to minimality. Thus, $\{\Psi_1,\dots,\Psi_{2n}\}$ is indeed a standard solution set. \par
In the second direction, a factorization $F(p)=(F_-\star D\star F_+)(p)$ induces a factorization $\omega(F)=\omega(F_-)\omega(D)\omega(F_+)$. By Theorem 3.8,  $\Psi_1,\dots,\Psi_{2n}$ given by 
$$
[\Psi_{1,+},\dots,\Psi_{2n,+}]=\omega(F_+)^{-1}, \quad [\Psi_{1,-},\dots,\Psi_{2n,-}]=\omega(F_-)\omega(D)
$$ 
form a complete and standard solution set for the barrier problem associated with $\omega(F)$. Then $\Phi_1,\dots,\Phi_n$ given by 
$$[\Phi_{1,+},\dots,\Phi_{n,+}]=F_+^{-\star}, \quad [\Phi_{1,-},\dots,\Phi_{n,-}]=F_-\star D
$$ 
satisfy $\omega[\Phi_1,\dots,\Phi_n]=[\Psi_1,\dots,\Psi_{2n}]$. As already seen, $\Psi_1,\dots,\Psi_{2n}$ being a standard set implies that $\Phi_1,\dots,\Phi_n$ are a standard set. Completeness follows similarly: Let $\Phi$ be a solution of the barrier problem associated with $F$ and define $\Psi=\omega(\Phi)$. There exist polynomials $r_1,\dots,r_{2n}$ (with coefficients in $\mathbb C_i$) such that 
$$
\Psi(z)=[\Psi_1(z),\dots,\Psi_{2n}(z)]\left[\begin{matrix}
r_1(z) &  r_{n+1}(z) \\
\vdots &\vdots \\
r_n(z) &  r_{2n}(z) \\
-\overline{r_{n+1}(\overline z)} & \overline{r_1(\overline z)} \\
\vdots &\vdots \\
-\overline{r_{2n}(\overline z)} &\overline{r_n(\overline z)}\\
\end{matrix}\right]
$$
$$
\left(\Psi(t)=[\Psi_1(t),\dots,\Psi_{2n}(t)]\left[\begin{matrix}
r_1\left(\dfrac{t-i}{t+i}\right) & r_{n+1}\left(\dfrac{t-i}{t+i}\right) \\
\vdots &\vdots \\
r_n\left(\dfrac{t-i}{t+i}\right) & r_{2n}\left(\dfrac{t-i}{t+i}\right) \\
-\overline{r_{n+1}\left(\dfrac{-t-i}{-t+i}\right)} & \overline{r_1\left(\dfrac{-t-i}{-t+i}\right)} \\
\vdots &\vdots \\
-\overline{r_{2n}\left(\dfrac{-t-i}{-t+i}\right)} & \overline{r_n\left(\dfrac{-t-i}{-t+i}\right)} \\
\end{matrix}\right]\right).
$$
This implies that 
$$
\Phi(p)=\sum_{k=1}^n\Phi_k(p)(r_k(p)+r_{n+k}(p)j)
$$
$$
\left(\Phi(p)=\sum_{k=1}^n\Phi_k(p)\left(r_k\left(\dfrac{p+1}{p-1}\right)+r_{n+k}\left(\dfrac{p+1}{p-1}\right)j\right)\right).
$$
\end{proof}

\begin{corollary}
Any standard solution set of eq. (3.13) (eq. (3.15)) is complete.
\end{corollary}

\begin{remark}
It will be interesting to consider more general quaternionic Riemann-Hilbert problems. Which kinds of surfaces in $\mathbb H$ may work other than spheres and hyperplanes? Which kinds of algebras of slice hyperholomorphic functions? A reasonable requirement is that the surface should split $\mathbb H$ into two axially symmetric s-domains.
\end{remark}

\section{Solvability of two classes of quaternionic functional equations}

The aim is to use factorization to characterize solvability of two classes of quaternionic functional equations. See \cite{gf} for the results in the complex case.

\begin{e-definition}
\hfill \break
\begin{enumerate}
\item[(i)] Let us denote by $W_d$ the set of all operators $A$ operating in the space $L_p[(0,\infty),\mathbb H]$ ($p\geq 1$) in accordance with the formula 
\begin{equation}
(A\phi)(t)=\sum_{n=-\infty}^{\infty}a_n\phi(t-n),
\end{equation}
where the $a_n\in\mathbb H$ satisfy $\sum_{n=-\infty}^{\infty}|a_n|<\infty$ (we define $\phi(t)=0$ for $t\leq 0$). Let us associate with each operator $A\in B_d$ the symbol function 
$$
\mathcal A(p)=\sum_{n=-\infty}^{\infty}p^na_n.
$$
\item[(ii)] Let us denote by $W_c$ the set of all operators $B$ operating in the space $L_p[(0,\infty),\mathbb H]$ in accordance with the formula 
\begin{equation}
(B\phi)(t)=c\phi(t)+\int_{0}^{\infty}k(t-s)\phi(s)ds,
\end{equation}
where $c\in\mathbb H, k(t)\in L_1(\mathbb R,\mathbb H)$. Let us associate with each operator \newline $B\in W_d$ the symbol function 
$$
\mathcal B(it)=c+\int_{-\infty}^{\infty}e^{itu}k(u)du.
$$
\end{enumerate}
\end{e-definition}

\begin{proposition}
A one-to-one correspondence exists between operators in $W_d$ ($W_c$) and their symbols in $\mathcal W_{\mathbb H}$ ($W(\mathbb S\mathbb R,\mathbb H)$). It is linear (with respect to right multiplication) and multiplicative in the following sense: If $A_-,A,A_+\in W_d$ ($B_-,B,B_+\in W_c$) are such that their symbols $\mathcal A_-,\mathcal A_+$ satisfy $\mathcal A_\pm\in\mathcal W_{\mathbb H,\pm}$ ($\mathcal B_-,\mathcal B_+$ satisfy $\mathcal B_\pm\in\mathcal W_\pm(\mathbb S\mathbb R,\mathbb H)$), then the operator $A_-AA_+$  ($B_-BB_+$) has the symbol $\mathcal A_-\mathcal A\mathcal A_+$ ($\mathcal B_-\mathcal B\mathcal B_+$).
\end{proposition}

\begin{proof}
The correspondence is well-defined, since an operator $A\in W_d$ \newline ($B\in W_c$)  is uniquely determined by $\{a_j\}_{j\in\mathbb Z}$ ($c\in\mathbb H$ and $k\in\L_1(\mathbb R,\mathbb H)$). Indeed, taking the function (for $r>0$)
$$
e_r(t)=\begin{cases}
    e^{-rt}, & \text{if $t>0$}\\
    0, & \text{if $t\leq0$}
  \end{cases},
$$
we get 
$$
(Ae)(t)=\sum_{n=-\infty}^{\infty}a_ne_r(t-n)
$$ 
for $t>0$. Then $Ae$ is piecewise differentiable with one-sided derivatives at $\mathbb N$ satisfying 
$$
a_n=-\dfrac{1}{r}((Ae_r)'(n^+)-(Ae_r)'(n^-))
$$ 
for all $n\in\mathbb N$. Assuming $A=0$, we get $a_n=0$ for $n\in\mathbb N$ and $$\sum_{n=-\infty}^{0}a_ne^{-r(t-n)}=0 \quad \forall r>0.$$ Equivalently, the power series $\sum_{m=0}^{\infty}p^ma_{-m}$ (which defines a slice hyperholomorphic function on $\mathbb B$) vanishes for $p\in (0,1)$, implying $a_{-m}=0$ for all $m\geq 0$. In the continuous case we have for $t>0$
$$
(Be_r)(t)=ce_r(t)+\int_{-\infty}^{t}k(s)e_r(t-s)ds=ce^{-rt}+e^{-rt}\int_{-\infty}^{t}k(s)e^{rs}ds,
$$
which implies that $Be_r$ is differentiable almost everywhere and satisfies
$$
k(t)=(Be_r)'(t)+r(Be_r)(t).
$$
Thus, if $B=0$, then $k(t)=0$ for $t>0$ and
$$
c+\int_{-\infty}^{0}k(s)e^{rs}ds=0
$$
for all $r>0$. Taking $r\rightarrow\infty$, the Riemann-Lebesgue lemma gives us $c=0$. Then the Laplace transform of $\tilde{k}(t)=k(-t)\max\{0,t\}$ vanishes on $(0,\infty)$, which implies $k(t)=0$ for $t\leq 0$. \newline
$A\mapsto\mathcal A$ is a bijection, since $\{a_j\}_{j\in\mathbb Z}$ uniquely determines a function in $\mathcal W_{\mathbb H}$. Likewise, $B\mapsto\mathcal B$ is a bijection since $c\in\mathbb H$ and $k\in\L_1(\mathbb R,\mathbb H)$ uniquely determine a function in $\mathcal W(\mathbb S\mathbb R,\mathbb H)$. Linearity of the correspondence is obvious, so it remains to check multiplicativity (in the sense defined). In the discrete case, we write
$$
(A_-\phi)(t)=\sum_{m=-\infty}^{0}b_m\phi(t-m) \thinspace , \thinspace (A_+\phi)(t)=\sum_{l=0}^{\infty}c_l\phi(t-l),
$$
and get
$$
(A_-AA_+)(\phi)(t)=\sum_{m=-\infty}^{0}\sum_{n=-\infty}^{\infty}\sum_{l=0}^{\infty}b_ma_nc_l\phi(t-l-n-m),
$$
for $t>0$, since for $l\geq0$ the $n$-shift of the (right) $l$-shift is the same as the $(n+l)$-shift, and for $m\leq 0$ the (left) $m$-shift of the $n$-shift is the same as the $(m+n)$-shift. Thus the corresponding symbol is
$$
\sum_{m=-\infty}^{0}\sum_{n=-\infty}^{\infty}\sum_{l=0}^{\infty}p^{m+n+l}b_ma_nc_l=\mathcal A_-\star\mathcal A\star\mathcal A_+.\\
$$
The continuous case is checked similarly.
\end{proof}

By the results of the preceding section (looking at the scalar-valued case), an invertible function in $\mathcal W_{\mathbb H}$ admits a factorization, as does an invertible function in $\mathcal W(\mathbb S\mathbb R,\mathbb H)$. While an operator is not necessarily invertible if its symbol is (due to the lack of general multiplicativity), Proposition 4.2 shows that it is invertible if its symbol is invertible in $\mathcal W_{\mathbb H,\pm}$ in the discrete case, or in $\mathcal W_{\pm}(\mathbb S\mathbb R,\mathbb H)$ in the continuous case.  The factorization $\mathcal A(p)=\mathcal A_-\star p^k\star\mathcal A_+$ induces the factorization $A=A_-U^{(k)}A_+$, where $U^{(k)}$ is the $k$-shift operator whose symbol is $p^k$. More precisely, for $k\geq0$ we have $U^{(k)}=U^k$, where $U$ is the right shift operator given by
$$
(U\phi)(t)=\begin{cases}
\phi(t-1) & \text{if $t>1$}\\
0 & \text{if $t\leq1$}
\end{cases}.
$$
For $k<0$ we have $U^{(k)}=(U^{(-1)})^{-k}$, where $U^{(-1)}$ is the left shift operator given by 
$$
(U^{(-1)}\phi)(t)=\begin{cases}
\phi(t+1) & \text{if $t>0$}\\
0 & \text{if $t\leq0$}
\end{cases}.
$$
$U^n$ is left invertible for $n\in\mathbb N$, satisfying $U^{(-n)}U^n=I$.\newline
Similarly, the factorization $\mathcal B(p)=\mathcal B_-\star \left(\dfrac{p+1}{p-1}\right)^m\star\mathcal B_+$ induces a factorization $B=B_-V^{(m)}B_+$, where
$$
(V\phi)(t)=\phi(t)-2\int_{0}^{t}e^{s-t}\phi(s)ds,\quad (V^{(-1)}\phi)(t)=\phi(t)-2\int_{t}^{\infty}e^{t-s}\phi(s)ds,
$$
and 
$$
V^{(m)}=\begin{cases}
V^m & \text{if $m\geq0$}\\
(V^{(-1)})^{-m} & \text{if $m<0$}
\end{cases}.
$$
The symbol of $V^{(m)}$ is $\left(\dfrac{p+1}{p-1}\right)^m$, since it is known to be $\left(\dfrac{t-i}{t+i}\right)^m$ in the complex case, and $p=it$. The operator $V^n$ is left invertible for $n\in\mathbb N$, satisfying $V^{(-n)}V^n=I$.

\begin{theorem}
Let $A\in W_d$ be such that the symbol is invertible and (thus) admits a factorization $\mathcal A(p)=\mathcal A_-(p)\star p^k\star\mathcal A_+(p)$. Then there are three cases:
\begin{enumerate} 
\item[(i)] If $k>0$, then $A$ is left invertible and the equation $A\psi=g$ is solvable if and only if the function $A_-^{-1}g$ vanishes on the segment $(0,k]$, where $A_-$ is the operator whose symbol is $\mathcal A_-$.
\item[(ii)] If $k=0$, then $A$ is invertible and thus the equation $A\psi=g$ always admits a unique solution.
\item[(iii)] If $k<0$, then $A$ is right invertible and every solution $\psi\in L_p[(0,\infty),\mathbb H]$ of the homogeneous equation $A\psi=0$ has the form $\psi=A_+^{-1}g$, where $g\in L_p[(0,\infty),\mathbb H]$ vanishes on $(-k,\infty)$ and $A_+$ is the operator whose symbol is $\mathcal A_+$.
\end{enumerate}
\end{theorem}

\begin{proof}
Going case by case:
\begin{enumerate} 
\item[(i)] If $k>0$, then $U^k$ is left invertible and since $A_-,A_+$ are invertible, it follows that $A=A_-U^kA_+$ is left invertible. The equation $A\psi=g$ is equivalent to $U^kA_+f=A_-^{-1}g$. The image of $U^k$ (thus of $U_kA_+$) is the set of all functions in $L_p((0,\infty),\mathbb H)$ that vanish on $(0,k]$. So $A\psi=g$ is solvable if and only if $A_-^{-1}g$ vanishes on $(0,k]$. \newline
\item[(ii)] If $k=0$, then $A=A_-A_+$ is invertible as the composition of invertible operators. \newline
\item[(iii)] If $k<0$, then $U^{(k)}$ being right invertible implies that $A$ is right invertible. The equation $A\psi=0$ is equivalent to $U^{(k)}\phi$=0, where $\phi=A_+\psi$. Since $U_k$ is a (left) $k$-shift, $U^{(k)}\phi=0$ if and only if $\phi=A_+\psi$ vanishes on $(-k,\infty)$.
\end{enumerate}
\end{proof}

\begin{theorem}
Let $B\in W_c$ be such that the symbol is invertible and (thus) admits a factorization 
$$
\mathcal B(p)=\mathcal B_-(p)\star \left(\dfrac{p+1}{p-1}\right)^m\star\mathcal B_+(p).
$$ 
Then there are three cases:
\begin{enumerate} 
\item[(i)] If $m>0$, then $B$ is left invertible and the equation $B\psi=g$ is solvable if and only if the function $B_-^{-1}g$ satisfies 
$$
\int_{0}^{\infty}(B_-^{-1}g)(t)t^ke^{-t}dt=0 \quad (k=0,1,\dots,m-1),
$$
where $B_-$ is the operator whose symbol is $\mathcal B_-$.
\item[(ii)] If $m=0$, then $B$ is invertible and thus the equation $B\psi=g$ always admits a unique solution.
\item[(iii)] If $m<0$, then $B$ is right invertible and every solution $\psi\in L_p[(0,\infty),\mathbb H]$ of the homogeneous equation $B\psi=0$ has the form 
$$
\psi(t)=B_+^{-1}\left(\sum_{j=0}^{-m-1}t^je^{-t}c_j\right),
$$ 
where $B_+$ is the operator whose symbol is $\mathcal B_+$ and the $c_j$ are arbitrary quaternions.
\end{enumerate}
\end{theorem}

For the proof, we need to characterize the image of $V^m$ for $m>0$, and the kernel of $V^{(m)}$ for $m<0$.

\begin{lemma}
For $m>0$, $\Ima V^m$ is the set of all functions $g\in L_p[(0,\infty),\mathbb H]$ satisfying 
$$
\int_{0}^{\infty}g(t)t^ke^{-t}dt=0
$$ 
for $k=0,1,\dots,m-1$. For $m<0$, $\Ker V^{(m)}$ is the set of all functions of the form $\sum_{j=0}^{-m-1}t^je^{-t}c_j$, where the $c_j$ are arbitrary quaternions.
\end{lemma}

\begin{proof}
For $m>0$, the operator $V^m$ maps $L_p[(0,\infty),\mathbb R]$ into itself, and it is known that the restriction to $L_p[(0,\infty),\mathbb R]$ has as its image the set of all functions $g\in L_p[(0,\infty),\mathbb R]$ satisfying 
$$
\int_{0}^{\infty}g(t)t^ke^{-t}dt=0
$$ 
for $k=0,1,\dots,m-1$. For a general $g\in L_p[(0,\infty),\mathbb H]$, we may write it as 
$$
g(t)=g_0(t)+g_1(t)e_1+g_2(t)e_2+g_3(t)e_3,
$$ 
where $g_1,g_2,g_3\in L_p[(0,\infty),\mathbb R]$, and $e_1,e_2,e_3\in\mathbb S$ are fixed, pairwise orthogonal quaternions satisfying $e_3=e_1e_2$. Then we deduce
\begin{equation*}
\begin{split}
&g\in \Ima V^m \iff \forall l=0,1,2,3, \quad g_l \in V^m(L_p[(0,\infty),\mathbb R]) \\
&\iff \forall l=0,1,2,3, \quad \forall k=0,1,\dots,m-1, \quad \int_{0}^{\infty}g_l(t)t^ke^{-t}dt=0 \\
&\iff \forall k=0,1,\dots,m-1, \quad \int_{0}^{\infty}g(t)t^ke^{-t}dt=0.
\end{split}
\end{equation*}
Similarly, for $m<0$ the operator $V^{(-m)}$ maps $L_p[(0,\infty),\mathbb R]$ into itself, and it is known that the kernel of its restriction to $L_p[(0,\infty),\mathbb R]$ is $\text{Span}_{\mathbb R}\{t^je^{-t}\}_{j=0}^{-m-1}$. Thus, its kernel is $\text{Span}_{\mathbb H}\{t^je^{-t}\}_{j=0}^{-m-1}$.
\end{proof}

The lemma readily leads to the proof of Theorem 4.4:

\begin{proof}
Going case by case:
\begin{enumerate} 
\item[(i)] If $m>0$, then $V^m$ being left invertible implies that $B=B_-V^mB_+$ is left invertible. The equation $B\psi=g$ is equivalent to $V^mB_+f=B_-^{-1}g$. By the lemma, $B\psi=g$ is solvable if and only if 
$$
\int_{0}^{\infty}(B_-)^{-1}(t)t^ke^{-t}dt=0
$$ 
for all $k=0,1,\dots,m-1$.\newline
\item[(ii)] If $m=0$, then $B=B_-B_+$ is invertible. \newline
\item[(iii)] If $m<0$, then $V^{(m)}$ being right invertible implies that $B$ is right invertible. The equation $B\psi=0$ is equivalent to $V^{(m)}\phi$=0, where $\phi=B_+\psi$. By the lemma, $V^{(m)}\phi=0$ if and only if $\phi=B_+\psi$ is of the form $\sum_{j=0}^{-m-1}t^je^{-t}c_j$, where the $c_j$ are arbitrary quaternions.
\end{enumerate}
\end{proof}

\begin{remark}
To calculate the factorization index $k$ of $f\in \mathcal W_{\mathbb H}$, we may fix $i\perp j\in S$ and use the formula
$$
k=\dfrac{1}{4\pi i}\int_{\partial\mathbb B\cap\mathbb C_i}\dfrac{g'(z)}{g(z)}dz,
$$
where $g(z)=\det(\omega_{i,j}(f)(z))$. This follows from the fact that $\omega_{i,j}(f)$ has index set $[k,k]$. Similarly, to the calculate the factorization index $m$ of $h\in\mathcal W(\mathbb S\mathbb R,\mathbb H)$, we use the formula
$$
m=\dfrac{1}{4\pi}[\text{arg}(r(t))]_{-\infty}^{\infty},
$$
where $r(t)=\det(\omega_{i,j}(h)(t))$, and $\text{arg}(r(t))$ is a continuous function describing the argument of $r(t)$ in the complex plane $\mathbb C_i$.
\end{remark}

\begin{remark}
There is a more general algebra of Wiener-Hopf integro-difference operators, for which the corresponding symbols are in $\mathcal B(\mathbb S\mathbb R,\mathbb H)$. The operators are of the form
$$
(C\phi)(t)=\sum_{u\in\mathbb R}c_u\phi(t-u)+\int_{0}^{\infty}k(t-s)\phi(s)ds,
$$
where $c_u\in\mathbb H$ vanish outside a countable set. While the correspondence between an operator $C$ and its symbol is one-to-one, linear and multiplicative in a certain sense, the difficulty lies in the lack of a (proven) factorization theorem for $\mathcal B(\mathbb S\mathbb R,\mathbb H)$. The characterization of the complex case relies on the factorization for $\mathcal B(\mathbb R,\mathbb C)$. It may be possible to show that the subalgebra $\omega(\mathcal B(\mathbb S\mathbb R,\mathbb H))$ fulfills a factorization theorem, even though $\mathcal B^{2\times 2}(\mathbb R,\mathbb C)$ does not.
\end{remark}

\section{Rational matrix functions and canonical factorization}

As defined in \cite{acs}, an $n\times n$ slice hyperholomorphic rational matrix function is obtained by a finite number of addition, $\star$-multiplication and $\star$-division operations on quaternionic polynomial matrix functions. Equivalently, the function is of the form $\dfrac{1}{r(p)}Q(p)$, where $Q(p)$ is a polynomial matrix function with quaternionic coefficients, and $r$ is a polynomial with real coefficients. For brevity, we will simply use the term rational matrix functions.

\begin{proposition}
Let $F$ be a rational matrix function. Then $F$ has no poles on $\partial \mathbb B$ if and only if $F$ is in $\mathcal W_{\mathbb H}^{n\times n}$. 
\end{proposition}

\begin{proof}
If $F(p)= \dfrac{1}{r(p)}Q(p)$ has no poles $\partial \mathbb B$, then without loss of generality $1/r(p)$ has no poles there. Indeed, since $r(p)$ has real coefficients, it has a factorization 
$$
r(p)=\prod_{i=1}^k(p^2-2\Rea(p_i)p+|p_i|^2)^{n_i} \prod_{j=1}^m(p-\alpha_j)^{t_j},
$$
where the $p_i$ lie on distinct spheres and the $\alpha_j$ are distinct real numbers. If any of the $p_i$ or $\alpha_j$ lies on $\partial \mathbb B$, then each entry of $Q(p)$ must be divisible by the corresponding polynomial 
$$
(p^2-2\Rea(p_i)p+|p_i|^2)^{n_i} \quad\mbox{or}\quad (p-\alpha_j)^{t_j},
$$ 
and we may divide both $Q(p)$ and $r(p)$ accordingly, getting rid of the pole. So we may assume that $r(p)$ does not vanish on $\partial \mathbb B$, which by the Wiener-L\'evy theorem, implies $1/r(p) \in \mathcal W_{\mathbb H}$ and consequently $F \in \mathcal W_{\mathbb H}^{n\times n}$. 

The other direction is obvious since functions in $\mathcal W_{\mathbb H}^{n\times n}$ are continuous on $\partial \mathbb B$.
\end{proof}

For our main theorem, we need to first show the existence of a certain type of realization for a general rational matrix function. The complex case (see \cite{bgkr}) is naturally extended as follows:

\begin{theorem}

Let $F$ be a rational $n\times n$ matrix function. Given a constant matrix $D\in\mathbb H^{n\times n}$, $F$ admits a realization of the form
\begin{equation}
F(p)=D+C\star(pG-A)^{-\star}B,
\end{equation}
where $B\in\mathbb H^{n\times m},C\in\mathbb H^{m\times n},G,A\in\mathbb H^{m\times m}$ are constant matrices.
\end{theorem}

\begin{remark}
The realization discussed in \cite{acs}, of the form 
$$
F(p)=\tilde D+p\tilde C\star (I-p\tilde A)^{-\star}\tilde B,
$$
assumes that $F$ is slice hyperholomorphic in a neighborhood of the origin (but need not be a square matrix). The theorem above has the restriction of $F$ being a square rational matrix function, but nothing else is assumed.
\end{remark}

\begin{proof}

The proof is largely unchanged from that of the complex case, but we will make sure that the arguments are still valid. First, $F(p)$ admits the following decomposition:
\begin{equation}
F(p)=K(p)+L(p),
\end{equation}
where $K(p)$ is slice hyperholomorphic at $\infty$ (that is, a proper rational matrix function, all of whose entries having a greater degree in the denominator than in the nominator) and $L(p)$  is a polynomial matrix. This follows from applying Euclid's division algorithm to each entry of $F(p)$; it should be noted that each entry admits the form $\dfrac{q(p)}{r(p)}$, where $r$ is a polynomial with real coefficients, and thus the algorithm may be applied as in a field (as any polynomial commutes with $r(p)$).
\newline
Without loss of generality, $K(\infty)=D$ (we can subtract a constant matrix from $K(p)$ and add it to $L(p)$). Let $G(p):=K(p^{-1})$. Since $G(0)=D$ and $G$ is defined at 0, it admits a realization of the form 
$$
G(p)=D+pC_K\star (I-pA_K)^{-\star}B_K,
$$
which yields 
$$
K(p)=D+p^{-1}C_K\star (I-p^{-1}A_K)^{-\star}B_K=D+C_K\star (pI-A_K)^{-\star}B_K.
$$
As for $L$, being a polynomial matrix we can write it as $$L(p)=L_0+pL_1+\dots+p^qL_q,$$ where $L_0,\dots,L_q$ are constant matrices. We introduce
\[
G_L=\left[\begin{matrix}
0  & I_m   &     \\
   & 0 & \ddots &    \\
  &  & \ddots & I_m \\
  &  &  & 0
\end{matrix}\right], \quad
B_L=\left[\begin{matrix}
L_0    \\
L_1    \\
\vdots \\
L_q
\end{matrix}\right], \quad
C_L=\left[\begin{matrix}
-I_m &0 &\cdots &0
\end{matrix}\right].
\]
The matrix $G_L$ is square of size $l=m(q+1)$, and is nilpotent of order $q+1$. It is easy to see that $$(I_l-pG_L)^{-\star}=\sum_{k=0}^{q}p^kG_L^k.$$ It follows that
$$
C_L\star (pG_L-I_l)^{-\star}B_L=-\sum_{k=0}^{q}p^kC_LG_L^kB_L=-\sum_{k=0}^{q}p^k(-L_k)=L(p).
$$
Finally, we can obtain the realization (5.1) for $F(p)$ by taking
\[
A=\left[\begin{matrix}
A_K & 0    \\
0 & I_l
\end{matrix}\right], \quad
B=\left[\begin{matrix}
B_K   \\
B_L
\end{matrix}\right], \quad
C=\left[\begin{matrix}
C_K & C_L 
\end{matrix}\right], \quad
G=\left[\begin{matrix}
I & 0    \\
0 & G_L
\end{matrix}\right].
\]
Here $I$ is of the same size as $A_K$.
\end{proof}

\begin{theorem}
Let $F$ be a rational $n\times n$ matrix function without poles on $\partial\mathbb B$ given by the realization
\begin{equation}
F(p)=I_n+C\star(pG-A)^{-\star}B.
\end{equation}
Put $A^{\times}=A-BC$. Then $F$ admits a canonical factorization if and only if the following two conditions are satisfied:
\begin{enumerate}
\item[(i)] $(pG-A^{\times})^{-\star}$ has no poles on $\partial\mathbb B$.
\item[(ii)] $\mathbb H^m=\Ima Q\oplus \Ker Q^{\times}$ and $\mathbb H^m=\Ima P\oplus \Ker P^{\times}$.
\end{enumerate}
Here $m$ is the order of the matrices $G$ and $A$, and
\begin{equation}
\begin{aligned}
Q=\dfrac{1}{2\pi}\int_{0}^{2\pi}e^{it}(pG-A)^{-\star}(e^{it})Gdt,\\ 
P=\dfrac{1}{2\pi}\int_{0}^{2\pi}Ge^{it}(pG-A)^{-\star}(e^{it})dt, \\
Q^{\times}=\dfrac{1}{2\pi}\int_{0}^{2\pi}e^{it}(pG-A^{\times})^{-\star}(e^{it})Gdt, \\
P^{\times}=\dfrac{1}{2\pi}\int_{0}^{2\pi}Ge^{it}(pG-A^{\times})^{-\star}(e^{it})dt.
\end{aligned}
\end{equation}
The above formulae do not depend on the choice of $i\in\mathbb S$, and the equalities in (ii) are equivalent. Moreover, if a canonical factorization exists, then it can be obtained by taking
\begin{equation}
\begin{aligned}
F_-(p)&=I_n+C\star (pG-A)^{-\star}(I-\sigma)B, \\
F_+(p)&=I_n+C\tau\star (pG-A)^{-\star}B,
\end{aligned}
\end{equation}
with inverses given by
\begin{equation}
\begin{aligned}
F_-^{-\star}(p)&=I_n-C(I-\tau)\star (pG-A^{\times})^{-\star}B, \\
F_+^{-\star}(p)&=I_n-C\star (pG-A^{\times})^{-\star}\sigma B.
\end{aligned}
\end{equation}
Here $\tau$ is the projection of $\mathbb H^m$ along $\Ima Q$ onto $\Ker Q^{\times}$ and $\sigma$ is the projection along $\Ima P$ onto $\Ker P^{\times}$. 
\end{theorem}

\begin{remark}
The factors $F_-,F_+$ in any canonical factorization are uniquely determined up to the transformation $F_-C,C^{-1}\star F_+$, where $C\in\mathbb H^{n\times n}$ is constant and invertible. This follows from the same argument used in the scalar-valued case of Theorem 3.10.
\end{remark}

It should be noted that the known theorem (see \cite{gms}) for complex-valued rational matrix functions is almost the same (with $\star$-multiplication replaced by regular matrix multiplication, $\mathbb H^n$ replaced by $\mathbb C^n$, and $F$ being a rational $n\times n$ matrix function without poles on $\partial\mathbb D$). To reduce the quaternionic case to the complex one, we need two lemmas:

\begin{lemma}
Given $i\perp j \in \mathbb S$, define $\rho:\mathbb H^n\rightarrow \mathbb C^n$  by $\rho(w)=[u \quad -\overline v]^T$, where $u,v\in\mathbb C_i^n$ are such that $w=u+vj$. Then for any $A\in \mathbb H^{n\times n}$, $w\in \mathbb H^n$ we have $\rho(Aw)=\chi(A)\rho(w)$.
\end{lemma}

\begin{proof}
Since $\chi$ is multiplicative, we have
$$
\chi(Aw)=\chi(A)\chi(w)=\chi(A)\left[\begin{matrix}
u & v \\
-\overline v & \overline u
\end{matrix}\right].
$$
Taking the left columns of both sides, we get $\rho(Aw)=\chi(A)\rho(w)$.

\end{proof}

\begin{lemma}
Let $F\in \mathcal W_{\mathbb H}^{n\times n}$ and $i\perp j \in \mathbb S$. Then $F$ admits a canonical factorization if and only if $\omega(F)$ does.
\end{lemma}

\begin{proof}
Corollary 2.19 and Theorem 3.10 imply that $F$ and $\omega(F)$ are simultaneously factorizable. If $F$ admits a canonical factorization $F(p)=F_-(p)\star F_+(p)$, then $\omega(F)(z)=(\omega(F_-)\omega(F_+))(z)$ is a canonical factorization. In the other direction, let $[k_1,\dots,k_n]$ be the index set of $F$. Then the index set of $\omega(F)$ is $[k_1,k_1,\dots,k_n,k_n]$, which implies $k_1=\dots=k_n=0$ given the canonical factorization. Thus $F$ is canonically factorizable.
\end{proof} 

Then we are ready to prove Theorem 5.3:

\begin{proof}
Since $F$ has no poles on $\partial\mathbb B$, Proposition 5.1 implies that $F\in \mathcal W_{\mathbb H}^{n\times n}$. Thus $F'(z):=(\omega(F))(z)$ is well defined (picking arbitrary $i\perp j\in \mathbb S$), and further, it is a $2n\times 2n$ complex-valued rational matrix function without poles on $\partial\mathbb B \cap \mathbb C_i$. Applying $\omega$ on (5.3), we see that
\begin{equation}
F'(z)=I_{2n}+C'(zG'-A')^{-1}B',
\end{equation}
where $$A'=\chi(A),\quad B'=\chi(B),\quad C'=\chi(C),\quad G=\chi(G).$$ Let $A'^{\times}=A'-B'C'$. The theorem in the complex case stipulates that $F'$ admits a canonical factorization if and only if the following two conditions are satisfied:
\begin{enumerate}
\item[(iii)] $(zG-A'^{\times})^{-1}$ has no poles on $\partial\mathbb B \cap \mathbb C_i$.
\item[(iv)] $\mathbb C^{2m}=\Ima Q'\oplus \Ker Q'^{\times}$ and $\mathbb C^{2m}=\Ima P'\oplus \Ker P'^{\times}$,
\end{enumerate}
where
\begin{equation}
\begin{aligned}
Q'=\dfrac{1}{2\pi i}\int_{\partial\mathbb B \cap \mathbb C_i}(zG'-A')^{-1}G'dz, \\
P'=\dfrac{1}{2\pi i}\int_{\partial\mathbb B \cap \mathbb C_i}G'(zG'-A')^{-1}dz, \\
Q'^{\times}=\dfrac{1}{2\pi i}\int_{\partial\mathbb B \cap \mathbb C_i}(zG'-A'^{\times})^{-1}G'dz, \\
P'^{\times}=\dfrac{1}{2\pi i}\int_{\partial\mathbb B \cap \mathbb C_i}G'(zG'-A'^{\times})^{-1}dz.
\end{aligned}
\end{equation}
Due to Lemma 5.5, it suffices to show that (i) is equivalent to (iii) and that (ii) is equivalent to (iv). Since $A'^{\times}=\chi(A-BC)=\chi(A^{\times})$, we have $\omega(pG-A^{\times})(z)=zG'-A'^{\times}$. Then by Proposition 5.1 and Corollary 2.19, (iii) is indeed equivalent to (i).
The two conditions in (iv) are known to be equivalent. To connect them to (ii), we need to check that 
$$
Q'=\chi(Q),\quad P'=\chi(P),\quad Q'^{\times}=\chi(Q^{\times}),\quad P'=\chi(P^{\times}).
$$ 
Assuming (i), $(pG-A)^{-\star}$ is of the form $\sum_{k=-\infty}^{\infty}p^kB_k$, where the series is absolutely summable entry-wise. Then  $$(zG'-A')^{-1}=\sum_{k=-\infty}^{\infty}z^k\chi(B_k)$$ is also absolutely summable, and we can interchange integration and summation as follows:
$$
Q=\dfrac{1}{2\pi}\int_{0}^{2\pi}e^{it}(pG-A)^{-\star}(e^{it})Gdt=\sum_{k=-\infty}^{\infty}\dfrac{1}{2\pi}\int_{0}^{2\pi}e^{i(k+1)t}B_kGdt=B_{-1}G,
$$
\begin{equation*}
\begin{split}
Q'&=\dfrac{1}{2\pi i}\int_{\partial\mathbb B \cap \mathbb C_i}(zG'-A')^{-1}G'dz=\sum_{k=-\infty}^{\infty}\dfrac{1}{2\pi i}\int_{\partial\mathbb B \cap \mathbb C_i}z^k\chi(B_k)G'dz\\
&=\chi(B_{-1})G'=\chi(Q).
\end{split}
\end{equation*}
The other cases are shown similarly; the above calculation also shows that formulae (5.4) do not depend on the choice of $i \in \mathbb S$. Now, it follows that $\Ima Q'=\rho(\Ima Q)$, since $Q'\rho(w)=\chi(Q)\rho(w)=\rho(Qw)$ by Lemma 5.4 (and it is easy to see that $\rho$ is a bijection). Similarly, 
$$\Ima P'=\rho(\Ima P),\quad \Ker P'^{\times}=\rho(\Ker P^{\times}),\quad \Ker Q'^{\times}=\rho(\Ker Q^{\times}).
$$ Since $\rho:\mathbb H^m\rightarrow \mathbb C^{2m}$ is an additive isomorphism, we have
\begin{equation}
\begin{aligned}
\mathbb H^m=\Ima Q\oplus \Ker Q^{\times} \iff \mathbb C^{2m}=\Ima Q'\oplus \Ker Q'^{\times} \iff \\
\iff \mathbb C^{2m}=\Ima P'\oplus \Ker P'^{\times} \iff \mathbb H^m=\Ima P\oplus \Ker P^{\times}.
\end{aligned}
\end{equation}
To prove the formulae for $F_-,F_+$, first note that the theorem in the complex case implies 
\begin{equation}
\begin{aligned}
\omega(F_-)(z)&=I_{2n}+C'(zG-A)^{-1}(I-\sigma')B, \\
\omega(F_+)(z)&=I_{2n}+C\tau'(zG-A)^{-1}B,
\end{aligned}
\end{equation}
with inverses given by
\begin{equation}
\begin{aligned}
\omega(F_-)^{-1}(z)&=I_{2n}-C(I-\tau')(zG-A^{\times})^{-1}B, \\
\omega(F_+)^{-1}(z)&=I_{2n}-C(zG-A^{\times})^{-1}\sigma' B,
\end{aligned}
\end{equation}
where $\tau'$ is the projection of $\mathbb C^{2m}$ along $\Ima Q'$ onto $\Ker Q'^{\times}$ and $\sigma'$ is the projection along $\Ima P'$ onto $\Ker P'^{\times}$. Now $\tau'=\chi(\tau)$, $\sigma'=\chi(\sigma)$. To see this, let $w\in \mathbb H^m$ decompose as $w=w_1+w_2$, where $w_1\in \Ima Q, w_2\in \Ker Q^{\times}$. Then $$\rho(w)=\rho(w_1)+\rho(w_2),$$ where $\rho(w_1)\in \Ima Q', \rho(w_2)\in \Ker Q'^{\times}$. By definition, we have $\tau(w)=w_2$, $\tau'\rho(w)=\rho(w_2)$ and so
$$
\tau'\rho(w)=\rho(\tau(w))=\chi(\tau)\rho(w).
$$
Since $w$ is arbitrary (and $\rho$ is surjective), $\tau'=\chi(\tau)$. Similarly, $\sigma'=\chi(\sigma)$. Finally, formulae (5.5) and (5.6) follow from formulae (5.10) and (5.11) using the fact that $\omega$ is injective, additive and multiplicative.
\end{proof}

\bibliographystyle{plain}


\end{document}